\documentclass[11pt]{amsart}
\usepackage{amssymb}
\usepackage{amscd}
\usepackage{aliascnt}
\usepackage[all]{xy}
\usepackage{comment}

\numberwithin{equation}{section}

\def\today{\number\day\space\ifcase\month\or   January\or February\or
   March\or April\or May\or June\or   July\or August\or September\or
   October\or November\or December\fi\   \number\year}

\theoremstyle{definition}
\newtheorem{thm}{Theorem}[section]

\newaliascnt{lem}{thm}
\newtheorem{lem}[lem]{Lemma}
\aliascntresetthe{lem}

\newaliascnt{prp}{thm}
\newtheorem{prp}[prp]{Proposition}
\aliascntresetthe{prp}

\newaliascnt{dfn}{thm}
\newtheorem{dfn}[dfn]{Definition}
\aliascntresetthe{dfn}

\newaliascnt{cor}{thm}
\newtheorem{cor}[cor]{Corollary}
\aliascntresetthe{cor}

\newaliascnt{cnj}{thm}

\aliascntresetthe{cnj}

\newaliascnt{cnv}{thm}

\aliascntresetthe{cnv}

\newaliascnt{rmk}{thm}
\newtheorem{rmk}[rmk]{Remark}
\aliascntresetthe{rmk}

\newaliascnt{ntn}{thm}
\newtheorem{ntn}[ntn]{Notation}
\aliascntresetthe{ntn}

\newaliascnt{exa}{thm}
\newtheorem{exa}[exa]{Example}
\aliascntresetthe{exa}

\newaliascnt{pbm}{thm}

\aliascntresetthe{pbm}

\newaliascnt{wrn}{thm}

\aliascntresetthe{wrn}

\newaliascnt{qst}{thm}
\newtheorem{qst}[qst]{Question}
\aliascntresetthe{qst}

\newaliascnt{exr}{thm}

\aliascntresetthe{exr}

\usepackage[capitalize]{cleveref}
\crefname{thm}{Theorem}{Theorems}
\crefname{lem}{Lemma}{Lemmas}
\crefname{prp}{Proposition}{Propositions}
\crefname{cor}{Corollary}{Corollaries}
\crefname{rmk}{Remark}{Remarks}
\crefname{dfn}{Definition}{Definitions}
\crefname{exa}{Example}{Examples}
\crefname{equation}{Equation}{Equations}

\newcommand{\beq}{\begin{equation}}
\newcommand{\eeq}{\end{equation}}
\newcommand{\beqr}{\begin{eqnarray*}}
\newcommand{\eeqr}{\end{eqnarray*}}
\newcommand{\bal}{\begin{align*}}
\newcommand{\eal}{\end{align*}}
\newcommand{\bei}{\begin{itemize}}
\newcommand{\eei}{\end{itemize}}

\newcommand{\Q}{{\mathbb{Q}}}
\newcommand{\Z}{{\mathbb{Z}}}
\newcommand{\R}{{\mathbb{R}}}
\newcommand{\C}{{\mathbb{C}}}
\newcommand{\N}{{\mathbb{Z}}_{> 0}}
\newcommand{\Nz}{{\mathbb{Z}}_{\geq 0}}
\newcommand{\T}{{\mathbb{T}}}

\newcommand{\Ea}{E^{\text{alg}}}

\newcommand{\id}{{\operatorname{id}}}
\newcommand{\ev}{{\operatorname{ev}}}

\newcommand{\rank}{{\operatorname{rank}}}

\newcommand{\card}{{\operatorname{card}}}

\newcommand{\Avg}{{\operatorname{Avg}}}
\newcommand{\Fl}{{\operatorname{Fl}}}
\newcommand{\Gr}{{\operatorname{Gr}}}
\newcommand{\Stab}{{\operatorname{Stab}}}

\newcommand{\cC}{{\mathcal{C}}}

\newcommand{\vH}{\check{H}}

\newcommand{\andeqn}{\qquad {\mbox{and}} \qquad}

\newcommand{\ca}{C*-algebra}

\newcommand{\fl}{\operatorname{Fl}}
\newcommand{\fla}{\operatorname{Fl}^{\text{alg}}}

\newcommand{\Ch}{\operatorname{Ch}}

\newcommand{\alg}{\mathrm{alg}}

\title[Obstructions to homomorphisms]{Obstructions to homomorphisms between homogeneous C*-algebras}   

\author{Ilan Hirshberg}
\address{Department of Mathematics, Ben-Gurion University of the Negev,
	P.O.B. 653, Be'er Sheva 84105, Israel}

\author{N. Christopher Phillips}
\address{Department of Mathematics, University  of Oregon,
	Eugene OR 97403-1222, USA.}

\makeatletter
\@namedef{subjclassname@2020}{\textup{2020} Mathematics Subject Classification}
\makeatother
\subjclass[2020]{46L80, 19K14}

\thanks{This material is based upon work supported by the
	US National Science Foundation under
	Grant DMS-2400332, and by the
	US-Israel Binational Science Foundation.}

\date{14~July 2026}

\begin{document}

\begin{abstract}
We develop a method to find new obstructions to the existence of homomorphisms between homogeneous C*-algebras with prescribed behavior on $K$-theory. As an application, we give a complete answer to a problem posed by Blackadar in 1993 concerning the existence of unital homomorphisms between algebras of matrix-valued functions on even spheres which are injective on $K_0$: we show that if  $n,d,k,r\in\N$ and
$k\notin \{n,n+1,\ldots,n+d-1\}$,
then there is no unital homomorphism 
\[
\varphi\colon C(S^{2n},M_r) \to C(S^{2k},M_{dr})
\]
such that $\varphi_*$ is injective on $K_0$.
\end{abstract}

\maketitle

\indent

\section{Introduction}\label{Sec_5904_Intro}

\indent
We find a new obstruction to the existence of homomorphisms between
homogeneous C*-algebras with specified behavior on $K$-theory.  We use
this obstruction to prove nonexistence results in cases in which ordered
$K$-theory does not rule out existence.  The methods give a complete
answer to a question raised by Blackadar in 1993; see
\cite[Question~3.16 and Proposition~3.15]{BlackadarMUT} and the surrounding discussion.  Blackadar asked, in
particular, whether there are unital homomorphisms
\[
        C(S^{2n},M_{2n}) \to C(S^{4n},M_{4n})
\]
which are injective on $K_0$, with the idea of using modifications of them to construct
perforated simple AH algebras.
As explained in \cite[Proposition~3.15]{BlackadarMUT}, there is no obstruction
in ordered $K$-theory to the existence of such maps.
Examples of perforated simple AH algebras were eventually constructed by Villadsen in \cite{VilladsenPerforation} using different techniques.
However, Blackadar's problem remained open.

We show that such homomorphisms do not exist for $n>1$.  More generally, let
$K$ denote the algebra of compact operators on a separable infinite-dimensional
Hilbert space.  We prove (\cref{thm:rank-d-obstruction})  that if $n,d,k,r\in\N$ and
\[
        k\notin \{n,n+1,\ldots,n+d-1\},
\]
then there is no homomorphism
\[
        \varphi\colon C(S^{2n},M_r) \to C(S^{2k},K)
\]
such that $\varphi_*$ is injective on $K_0$ and $\varphi(1)$ has constant
rank $dr$.

 This obstruction is sharp. 
For $k\in\{n,n+1,\ldots,n+d-1\}$, Blackadar pointed out (\cite[page 34]{BlackadarMUT})
that unital homomorphisms $C(S^{2n},M_r) \to C(S^{2k},M_{dr})$
which are injective on $K_0$ do exist.
The argument is as follows.
For any $m \in \N$ there exists a unitary $u \in C(S^{2m-1},M_m)$
which generates $K_1 (C(S^{2m-1}))$.
As unitaries define maps from $C(S^1)$, this means that there exists a unital homomorphism from $C(S^1)$ to $C(S^{2m-1},M_m)$ which induces an isomorphism on $K_1$.
The $(2n-1)$-fold suspension map gives us a unital homomorphism from $C(S^{2n})$ to $C(S^{2m+2n-2},M_m)$ which induces an injective homomorphism on $K_0$. We can replace the matrix size $m$ by a larger matrix size by adding point evaluations down the diagonal. So setting $k = m+n-1$ shows us that such homomorphisms exist provided $k$ is in the specified range. (If such a homomorphism exists for $r=1$, we can get it for any $r$ by tensoring by $M_r$.)
 The same argument, using one suspension less, shows that if $k\in\{n,n+1,\ldots,n+d-1\}$ then there exists a unital homomorphism from $C(S^{2n-1})$ to $C(S^{2k-1},M_d)$ which is injective on $K_1$; we deduce the converse in \cref{cor:odd-sphere-obstruction}. 

 While our result shows that Blackadar's intended construction of perforated simple AH algebras cannot work, our work does reveal a new unstable phenomenon, and it now becomes interesting to see whether there are simple C*-algebras that exhibit it. See Question~\ref{qst:exotic_AH}.

In \cite{DadarlatNemethiShape}, D{\u{a}}d{\u{a}}rlat and N{\'e}methi studied related questions, which are complementary in a sense. Their work focused on the case in which the matrix sizes are large.
We denote the space of unital homomorphisms from $C(Y)$ to $M_d$ by $E_{d,Y,d}$. If we fix a base-point $y_0 \in Y$, we get a stabilization map 
$E_{d,Y,d} \to E_{d+1,Y,d+1}$
given by putting a point evaluation at $y_0$ in the bottom right corner. 
 In \cite[Theorem~6.4.2]{DadarlatNemethiShape}, it is shown that for a
connected finite CW complex $Y$, this stabilization map 
is a $2\lfloor d/3\rfloor$-homotopy equivalence for $d\ge 3$.
It follows from \cite[Corollary~6.4.4 and Remark~6.4.5]{DadarlatNemethiShape}
that, if $X$ is a finite CW complex with
$\dim(X)<2\lfloor d/3\rfloor$,
then the set $[C(Y),C(X,M_d)]_1$
of homotopy classes of unital homomorphisms from $C(Y)$ to $C(X,M_d)$
can be identified with the connective $kk$-theory group $kk(X,Y)$
defined in their paper.
Applying this with $Y=S^{2n}$ and $X=S^{2k}$, one can show, using the results of
\cite[Section~3.3]{DadarlatNemethiShape}, that there are
unital homomorphisms from $C(S^{2n})$ to $C(S^{2k},M_d)$
which are injective on $K_0$ whenever $k\geq n$ and
$d\ge 3k+3$.  In this case, Blackadar's suspension construction gives 
existence in the much larger range $k \geq n$ and $d \geq k-n+1$.

In contrast to the work of D{\u{a}}d{\u{a}}rlat and N{\'e}methi, we look at the unstable case, in which the matrix sizes are small. 
Let $X$ and $Y$ be compact Hausdorff
spaces.  We consider homomorphisms
\[
\varphi\colon C(Y) \to C(X,K)
\]
such that $\varphi(1)$ has constant rank $d$.  If we denote by $E_{d,Y}$ the space of homomorphisms from $C(Y)$ to $ K$ which send $1$ to a
projection of rank $d$, with the topology of pointwise norm convergence, such homomorphisms correspond
bijectively to continuous maps $h_\varphi\colon X\to  E_{d,Y}$ given by $h_\varphi(x)(f)=\varphi(f)(x)$.
When we specialize to $X = S^{2k}$ and $Y = S^{2n}$, the map $h_{\varphi} \colon S^{2k} \to E_{d,S^{2n}}$ defines a class in the homotopy group $\pi_{2k} ( E_{d,S^{2n}} )$. Understanding such maps up to homotopy means computing these homotopy groups. This is hopeless using the current techniques in algebraic topology,  as not even all the homotopy groups of spheres are known. However---although this doesn't appear explicitly in our techniques---we can show that if $\varphi_*$ is injective on $K_0$, then the class of $h_{\varphi}$ in $\pi_{2k} ( E_{d,S^{2n}} )$ is non-torsion. This means that it can be detected using techniques of rational homotopy theory, which make the problem much more accessible.
Indeed, our methods, which use indecomposability of elements in the rational cohomology ring, are informed by techniques of rational homotopy theory. To illustrate this in a very simple case, we know that $H^4(S^2 \times S^2) \cong H^4(S^4) \cong \Z$, but there is no map $S^4 \to S^2 \times S^2$ which is non-zero on $H^4$. The easiest way to see this is to note that $\pi_4(S^2 \times S^2) \cong \pi_4(S^2) \times \pi_4(S^2) \cong \Z_2 \times \Z_2$, while if there existed a continuous map that pulled back a non-zero element of $H^4(S^2 \times S^2)$ to a non-zero element of $H^4(S^4)$, then it would have to be non-torsion in $\pi_4$. However, we can also see this by considering the structure of the graded cohomology rings of these spaces. Let $\xi_1$ and $\xi_2$ be the generators of $H^2(S^2 \times S^2)$ gotten by pulling back a generator of $H^2(S^2)$ along the two coordinate projection maps. Then $\eta = \xi_1 \xi_2$ is a generator of $H^4(S^2 \times S^2)$.
If $f \colon S^4 \to S^2 \times S^2$ is continuous, then $f^*(\eta) = f^*(\xi_1) f^*(\xi_2) = 0$, as $H^2(S^4) = 0$. 

The space $E_{d,Y}$ is not compact, so there are some technical subtleties involving representable $K$-theory which are needed to handle it; we gloss over these in the introduction. It turns out that to address our problem, it is enough to consider rational $K$-theory, so the Chern character converts our problem to one in rational cohomology, where we can exploit the $\Z$-grading. An element of a graded $\Q$-algebra is called \emph{decomposable} if it is in the linear span of products of elements of strictly lower degree; otherwise it is \emph{indecomposable}. The image of a decomposable element under a graded ring homomorphism is never indecomposable. The space $E_{d,Y}$ is the quotient of a product of copies of $Y$ and an infinite flag manifold by a somewhat involved equivalence relation. We can nonetheless compute its rational cohomology ring. When $Y = S^{2n}$, for most choices of $k$, all of  $\vH^{2k}(E_{d,Y};\Q)$ is decomposable. In these cases, no map $h \colon S^{2k} \to E_{d,S^{2n}}$ can be nonzero on rational cohomology of degree $2k$. 

We briefly describe possible connections between \cref{Sec_remarks_and_problems}, in which we discuss possible generalizations of our results to approximate homomorphisms, and  recent work on approximate representations.
In \cite{DadarlatQuasiRepAlmostFlat}, D{\u{a}}d{\u{a}}rlat shows that sufficiently multiplicative
unital completely positive maps into matrix algebras can carry $K$-theoretic information that 
cannot be seen via finite-dimensional representations.  In
\cite[Definition~2.1 and Section~2]{DadarlatQuasiRepAlmostFlat}, he defines
a partially defined push-forward of
$K_0$-classes by approximate homomorphisms. For a large class of groups,
 completely positive asymptotic representations of $C^*(G)$ realize
prescribed homomorphisms from $K_0(C^*(G))$ to $\Z$ (\cite[Theorem~1.1]{DadarlatQuasiRepAlmostFlat}).  In subsequent work, D{\u{a}}d{\u{a}}rlat uses almost-flat $K$-theory classes which are not flat as obstructions to
matricial stability; see \cite[Theorems~1.1 and~1.2]{DadarlatMatricialStability}.
Enders and Shulman's characterization of matricial semiprojectivity for
commutative C*-algebras gives further information on when
approximate finite-dimensional representations are  close to genuine ones.
For finite-dimensional compact metrizable $X$, $C(X)$ is matricially
semiprojective if and only if $\dim(X)\leq 2$ and $H^2(X;\Q)=0$;
see \cite[Theorem~4.1]{EndersShulmanMatricial}. We discuss related open problems in \cref{Sec_remarks_and_problems}. 

The paper is organized as follows. In \cref{Sec_5904_Prelim} we fix notation and collect some facts on cohomology. In \cref{Sec_5904_Gen} we develop the general framework used for our obstructions. In \cref{Sec_5904_General_rank}, we analyze the rational cohomology rings of the spaces of homomorphisms, and use them to obtain our main theorem, \cref{thm:rank-d-obstruction}, which settles Blackadar's problem. We conclude with a few remarks and open problems in \cref{Sec_remarks_and_problems}.

We thank Andrew Toms, Ben Elias, Robert Lipshitz, Nicholas Proudfoot and especially Dev Sinha for helpful conversations and pointers. ChatGPT was used to help simplify and generalize one of the technical steps in our cohomology computations, to search for references, to suggest draft language for a few passages, and to proofread.

\section{Preliminaries}\label{Sec_5904_Prelim}

\indent
We collect here a number of results on cohomology which are needed for the proof of our main theorem.

\begin{ntn}
	For a C*-algebra $A$, we let $P(A)$ denote the space of projections in $A$, with the norm topology induced from $A$.
\end{ntn}

\begin{ntn}\label{ntn:M_infty}
	We write $M_\infty=\bigcup_{n=1}^{\infty}M_n$, equipped with the direct limit topology. 
\end{ntn}

The following result is used to compute the rational cohomology of quotients
by finite group actions.  It shows that, after inverting the order of the
finite group, the \v{C}ech cohomology of the quotient is the fixed point
subring of the cohomology of the original space.  This is false in general for
cohomology with integer coefficients or with $\Z_p$ coefficients.
The theorem can be found in \cite{Bredon}. Theorem~2.4 in Chapter III of \cite{Bredon} has the statement for homology for simplicial complexes, the discussion on the next page says that it applies to cohomology as well, and Theorem~7.2 in Chapter III of \cite{Bredon} extends this to  paracompact spaces. A more general version for sheaf cohomology can be deduced from \cite[Section 5]{grothendieck}.

\begin{thm}\label{thm:cohomology-quotient}
Let $X$ be a paracompact Hausdorff space, let $G$ be a finite group acting on
$X$, and let $F$ be a field whose characteristic does not divide $\card(G)$.
Then the quotient map $X\to X/G$ induces a natural isomorphism
\[
        \vH^*(X/G;F) \cong \vH^*(X;F)^G.
\]
\end{thm}

\begin{dfn}\label{dfn:indecomposable-elements}
By a \emph{graded ring} we mean a unital ring $R$ equipped with a decomposition as a direct sum of
abelian groups
\[
        R=\bigoplus_{n\in\Nz} R_n,
\]
with $R_kR_m\subseteq R_{k+m}$ for all $k,m\in\Nz$, and such that
$1\in R_0$.  If $R$ and $S$ are two graded rings, a \emph{graded homomorphism}
$\pi\colon R\to S$ is a homomorphism which satisfies $\pi(R_n)\subseteq S_n$
for all $n$.

For $n>0$, an element $a\in R_n$ is \emph{indecomposable} if it is not in the
subring generated by $R_0,R_1,\ldots,R_{n-1}$.  Equivalently, if
\[
        R_+=\bigoplus_{n=1}^{\infty} R_n,
\]
then $a$ is indecomposable if and only if its image in
$R_+/(R_+\cdot R_+)$ is nonzero.  When $R$ is an algebra over a field $F$, we
say that the indecomposables of degree $n$ are generated by a set
$T\subset R_n$ if the images of the elements of $T$ generate
$R_n/(R_+\cdot R_+)$ as an $F$-vector space.
\end{dfn}

The following lemma is immediate.
\begin{lem}\label{lemma:indecomposables_in_rings}
Let $R$ and $S$ be graded rings, and let $\pi\colon R\to S$ be a graded
homomorphism. Let $n>0$ and let $a\in R_n$. If $\pi(a)$ is indecomposable, then
so is $a$.
\end{lem}

We apply this to get the following straightforward corollary.
\begin{cor}\label{cor:indecomposable_for_maps_from_sphere}
Let $X$ be a compact Hausdorff space, and let
$\eta\in \vH^n(X;\Q)$ with $n>0$. If there exists a continuous map
$h\colon S^n\to X$ such that $h^*\eta\neq 0$, then $\eta$ is indecomposable.
\end{cor}
\begin{proof}
The positive degree part of $\vH^*(S^n;\Q)$ has square zero in degree $n$.
Thus every nonzero element of $\vH^n(S^n;\Q)$ is indecomposable, and the result
follows from \cref{lemma:indecomposables_in_rings}.
\end{proof}

We record the following elementary observation. It is well known and the proof 
is straightforward, so it is omitted.

\begin{lem}\label{lem:transitive-components}
	Let $G$ be a finite group acting on a topological space $Y$. Suppose $Y$ decomposes as a disjoint union of closed and open sets $Y=\coprod_{i\in I}Y_i$. Suppose $G$ acts transitively on $I$, and the quotient map $Y \to I$ is $G$-equivariant.   Fix $i_0\in I$, and set $H=\Stab_G(i_0)$.  Then the inclusion
	$Y_{i_0}\subset Y$ induces a homeomorphism
	\[
	Y_{i_0}/H\cong Y/G.
	\]
\end{lem}

For a finite sequence of integers $0=n_0<n_1<n_2<\cdots<n_k$, we
denote by $\fl(n_1,n_2,\ldots,n_k)$ the complex flag manifold consisting of
chains of subspaces
\[
        V_1\subset V_2\subset\cdots\subset V_k\subseteq \C^{n_k}
\]
such that $\dim(V_j)=n_j$.

This is a compact manifold.  It is convenient to describe it
using projections. We identify the flag manifold with tuples of mutually
orthogonal projections summing to the identity:
\[
\begin{aligned}
	\fl(n_1,n_2,\ldots,n_k)
	=\Bigl\{\, &(p_1,p_2,\ldots,p_k)\in P(M_{n_k})^k
	\ \Bigm|\ \sum_{j=1}^k p_j=1_{n_k},\\
	&\text{and } \rank(p_j)=n_j-n_{j-1}
	\text{ for } j=1,2,\ldots,k \Bigr\}.
\end{aligned}
\]
Here, the subspace $V_m$ in the flag corresponds to the range of the sum
$\sum_{j=1}^m p_j$.

When we allow the ambient space to be infinite dimensional, we need to specify
the topology more carefully. Let $H$ be a separable infinite-dimensional Hilbert
space. We define the infinite flag space corresponding to finite ranks
$n_1<\cdots<n_k$ inside $H$ as
\[
\begin{aligned}
	\fl(n_1,n_2,\ldots,n_k,\infty)
	=\Bigl\{\, &(p_1,p_2,\ldots,p_k)\in P(K(H))^k
	\ \Bigm|\ p_l p_j=0\text{ for }l \neq j,\\
	&\text{and } \rank(p_j)=n_j-n_{j-1}
	\text{ for } j=1,2,\ldots,k \Bigr\} ,
\end{aligned}
\]
with the norm topology inherited from $K(H)^k$.  Note that here $\sum_{j=1}^k p_j\neq 1$.

The standard algebraic topology model, using direct limits, is
\[
\begin{aligned}
	\fla(n_1,n_2,\ldots,n_k,\infty)
	=\Bigl\{\, &(p_1,p_2,\ldots,p_k)\in P(M_{\infty})^k
	\ \Bigm|\ p_l p_j=0\text{ for }l \neq j,\\
	&\text{and } \rank(p_j)=n_j-n_{j-1}
	\text{ for } j=1,2,\ldots,k \Bigr\}.
\end{aligned}
\]
Neither space is a
finite-dimensional manifold; we use the term ``flag manifold'' by abuse of
terminology. The standard inclusion $M_\infty\subset K(H)$ induces an inclusion
\[
        \iota\colon \fla(n_1,n_2,\ldots,n_k,\infty)\to
        \fl(n_1,n_2,\ldots,n_k,\infty).
\]

\begin{lem}\label{lem:alg_flag_manifold}
The standard inclusion
\[
        \iota\colon \fla(n_1,n_2,\ldots,n_k,\infty)\to
        \fl(n_1,n_2,\ldots,n_k,\infty)
\]
is a weak homotopy equivalence.
\end{lem}
\begin{proof}
We prove the relative approximation statement needed for weak homotopy groups.
Let $X$ be a compact Hausdorff space, let $A\subset X$ be closed, and let
$f\colon X\to \fl(n_1,\ldots,n_k,\infty)$ be continuous. Assume that
$f(A)$ is contained in the algebraic model and, as a compact subset of the
direct-limit space, is contained in one finite stage. We show that $f$ is
homotopic, relative to $A$, to a map with image in the algebraic model. The
case $X = S^r$, $x_0 \in S^r$ some fixed basepoint and $A=\{x_0\}$ gives surjectivity on homotopy groups, and the case
$X = S^r \times [0,1]$ and $A=( S^r\times\{0,1\} ) \cup ( \{x_0\} \times [0,1] )$ gives injectivity.

Map a $k$-tuple of mutually orthogonal projections $(p_1,\ldots,p_k)$ to
$\sum_{j=1}^k jp_j$.  This identifies
$\fl(n_1,\ldots,n_k,\infty)$ with the set of positive compact operators whose
spectrum is $\{0,1,2,\ldots,k\}$ and such that the dimension of the eigenspace
of $j$ is $n_j-n_{j-1}$ for $j=1,2,\ldots,k$.

Let $q_m$ be the orthogonal projection onto the span of the first $m$ standard
basis vectors of $H$. Since $X$ is compact,
\[
        \|q_m f(x)q_m-f(x)\|\longrightarrow 0
\]
uniformly on $X$. Choose $m$ large enough that this norm is less than $1/4$ for
all $x\in X$, and also large enough that $q_m f(x)q_m=f(x)$ for every $x\in A$.
Let
\[
        g\colon \R\setminus\{r+1/2 \mid r\in\Z\}\to \Z
\]
be the continuous function which sends a point to the nearest integer.  For
$t\in[0,1]$, set
\[
        a_t(x)=(1-t)f(x)+tq_mf(x)q_m.
\]
The spectrum of $a_t(x)$ does not contain any point $r+1/2$; hence we can define
$h_t(x)=g(a_t(x))$ using functional calculus.  Then $h_0=f$, the
homotopy is fixed on $A$, and $h_1(X)$ is contained in $q_mK(H)q_m\cong M_m$.
Since the spectral projections do not cross the gaps between adjacent integers,
their ranks are constant along the homotopy. Thus $h_t$ stays inside the flag
space throughout, and $h_1$ has image in the algebraic model.
\end{proof}

\begin{ntn}
	Let $E$ be a vector bundle over a topological space $X$. We denote by $c_n(E) \in \vH^{2n}(X)$ the $n$-th Chern class of $E$.
\end{ntn}

\begin{thm}\label{thm:stable-flag-cohomology}
The ranges of the projections $p_1,\ldots,p_k$ define canonical vector bundles
$E_1,\ldots,E_k$ over $\fl(n_1,\ldots,n_k,\infty)$. Let
$r_j=n_j-n_{j-1}=\rank(E_j)$. Then
\[
\vH^*(\fl(n_1,\ldots,n_k,\infty);\Q)
     \cong \Q[c_l(E_j):1\leq j\leq k \text{ and } 1\leq l\leq r_j],
\]
where $\deg(c_l(E_j))=2l$.
\end{thm}
\begin{proof}
	Let $m=n_k$. Let $V_m(H)$ denote the
space of orthonormal $m$-frames in $H$, with the norm topology. 
We claim that $V_m(H)$ is contractible. To see that, identify $H$
with $L^2([0,\infty))$, and let $(S_t)_{t\geq 0}$ be the unilateral shift
semigroup.
The maps 
$$
(\xi_1,\xi_2,\ldots,\xi_m)\mapsto (S_t \xi_1,S_t \xi_2,\ldots,S_t \xi_m)
$$ 
for $t \in [0,1]$ define a homotopy in $V_m(H)$ from the identity map to a map whose image is
$V_m(L^2([1,\infty)))$. Now, choose orthonormal vectors $e_1,\ldots,e_m$ in $L^2([0,1])$. For
$(\xi_1,\ldots,\xi_m)\in V_m(L^2([1,\infty)))$ and for $t \in [0,1]$, define
$$
T_t(\xi_1,\ldots,\xi_m)
=
\bigl(
\sqrt{t}\,e_1+\sqrt{1-t}\,\xi_1,
\ldots,
\sqrt{t}\,e_m+\sqrt{1-t}\,\xi_m
\bigr) .
$$
So, $T_t(\xi_1,\ldots,\xi_m)$ is an orthonormal $m$-frame for every
$t$. The map $T_0$ is the inclusion of $V_m(L^2([1,\infty)))$ into $V_m(H)$,
while the map $T_1$ is the constant map with value $(e_1,\ldots,e_m)$. Therefore
$V_m(H)$ is contractible, as claimed.

Set $r_j=n_j-n_{j-1}$ and $G=U(r_1)\times U(r_2) \times \cdots\times U(r_k)$. The group
$G$ acts on $V_m(H)$ blockwise. 
Note that we have
$$
\fl(n_1,\ldots,n_k,\infty) \cong V_{n_k}(H)/G .
$$
The quotient map
$V_{n_k}(H) \to V_{n_k}(H)/G
$
is a principal $G$-bundle. Since $V_{n_k}(H)$ is contractible, this quotient is a model
for $BG$.
Using \cite[Section 16.5]{May_book}, we have
$$
B(U(r_1)\times\cdots\times U(r_k))\cong BU(r_1)\times\cdots\times BU(r_k).
$$
Since the Grassmannian $\Gr(r,\infty)$ is a
model for $BU(r)$, we get
\[
\fl(n_1,\ldots,n_k,\infty)
\simeq \Gr(r_1,\infty)\times\cdots\times\Gr(r_k,\infty).
\]
The rational cohomology of $BU(r)$ is the polynomial algebra on the Chern classes
$c_1,\ldots,c_r$ \cite[Theorem 14.5]{MilnorStasheff}. The K\"unneth formula
then gives the stated polynomial algebra. Here we use the fact that, for spaces of CW
homotopy type, \v{C}ech cohomology agrees with singular cohomology.
\end{proof}

\begin{rmk}
	The proof of \cref{thm:stable-flag-cohomology} shows in particular that
	$\fl(n_1,\ldots,n_k,\infty)$ has the homotopy type of a CW complex. The
	algebraic flag manifold $\fla(n_1,\ldots,n_k,\infty)$ also has CW homotopy
	type, since it is the direct limit of finite dimensional flag manifolds.
	Therefore the weak homotopy equivalence
	$$
	\fla(n_1,\ldots,n_k,\infty) \to
	\fl(n_1,\ldots,n_k,\infty)
	$$
	from \cref{lem:alg_flag_manifold} is a homotopy equivalence by Whitehead's
	theorem. Consequently the conclusion of
	\cref{thm:stable-flag-cohomology} applies to
	$\fla(n_1,\ldots,n_k,\infty)$ as well.
\end{rmk}

\begin{prp}\label{prp:finite-complete-flag-cohomology}
For $s\in\N$, write
\[
        \Fl_s=\fl(1,2,\ldots,s)=U(s)/\T^s.
\]
Let $L_1,\ldots,L_s$ be the tautological line bundles over $\Fl_s$, and set
$z_i=c_1(L_i)\in H^2(\Fl_s;\Q)$. If $e_j$ denotes the $j$-th elementary
symmetric polynomial in $z_1,\ldots,z_s$, then
\begin{equation}\label{eq:finite-flag-cohomology}
        \vH^*(\Fl_s;\Q)
        \cong \Q[z_1,\ldots,z_s]/(e_1,\ldots,e_s) 
        \; \text{  with } \; \deg(z_i)=2.
\end{equation}
\end{prp}
\begin{proof}
The bundle $L_1\oplus\cdots\oplus L_s$ is the trivial rank $s$ bundle, so
\[
        \prod_{i=1}^s(1+z_i)=1,
\]
and hence $e_1=\cdots=e_s=0$. These are the complete relations for the full
flag manifold; see \cite[Proposition 21.17]{BottTu}. The statement in
\cite{BottTu} is written with real coefficients, but the same presentation over
$\Q$ follows in a similar manner.
\end{proof}

\begin{dfn}\label{dfn:forgetful_map}
For any sequence of dimensions $0<n_1<n_2<\cdots<n_k$ and any subsequence of
indices $1\leq m_1<m_2<\cdots<m_l=k$, let $N_j=n_{m_j}$. The canonical
\emph{forgetful map}
\[
\tau\colon \fl(n_1,\ldots,n_k,\infty)\to
       \fl(N_1,N_2,\ldots,N_l,\infty)
\]
is given by summing the orthogonal projections between the preserved indices:
\[
(p_1,\ldots,p_k)\mapsto
\left(\sum_{j=1}^{m_1}p_j,\sum_{j=m_1+1}^{m_2}p_j,\ldots,
      \sum_{j=m_{l-1}+1}^{m_l}p_j\right).
\]
In the finite case, where $n_k$ is the total dimension, the map is defined
analogously, provided the final ranks match, that is, provided $m_l=k$.
\end{dfn}


\begin{exa}\label{exa:forgetful_12}
Let $\tau\colon \fl(1,2,\infty)\to \fl(2,\infty)$ be the canonical forgetful map
given by $\tau(p_1,p_2)=p_1+p_2$. Let $z_1$ and $z_2$ be the first Chern classes of the tautological line bundles given by the images of the projections $p_1$ and $p_2$. Then we can identify
$\vH^*(\fl(1,2,\infty);\Q)\cong \Q[z_1,z_2]$, with
$\deg(z_1)=\deg(z_2)=2$. Let $E$ be the tautological rank $2$ vector bundle obtained as the image of $p_1+p_2$, let $x = c_1(E)$ and let $y = c_2(E)$. We then identify
$\vH^*(\fl(2,\infty);\Q)\cong \Q[x,y]$, with $\deg(x)=2$ and
$\deg(y)=4$. Then $\tau^*(x)=z_1+z_2$ and $\tau^*(y)=z_1z_2$.

Let $\mu\colon \fl(1,2,\infty)\to \fl(1,2,\infty)$ be the order two
automorphism given by $\mu(p_1,p_2)=(p_2,p_1)$. Then $\tau=\tau\circ\mu$, and
$\mu^*(z_1)=z_2$ and $\mu^*(z_2)=z_1$. The fixed point subring is
$\Q[z_1+z_2,z_1z_2]$, which is the image of $\tau^*$.
\end{exa}

\begin{exa}\label{exa:forgetful_123}
This principle generalizes directly to larger flag manifolds. Consider
$\fl(1,2,3,\infty)$, parameterized by three rank one projections
$(p_1,p_2,p_3)$, with cohomology
$\vH^*(\fl(1,2,3,\infty);\Q)\cong \Q[z_1,z_2,z_3]$.

\begin{enumerate}
\item The forgetful map
\[
\tau_{1,3}\colon \fl(1,2,3,\infty)\to \fl(1,3,\infty)
\] 
is given by
$(p_1,p_2,p_3)\mapsto(p_1,p_2+p_3)$. Write
\[
\vH^*(\fl(1,3,\infty);\Q) = \Q[w_1,x_1,x_2] ,
\]
where $w_1$ is the Chern
class from the rank one bundle and $x_1,x_2$ are the Chern classes from the
rank two bundle. Then
\[
\tau_{1,3}^*(w_1)=z_1,
\qquad \tau_{1,3}^*(x_1)=z_2+z_3,
\; \text{ and } \; \tau_{1,3}^*(x_2)=z_2z_3.
\]
The involution which exchanges $p_2$ and $p_3$ fixes the image of
$\tau_{1,3}^*$ and identifies it with
$\Q[z_1,z_2,z_3]^{\Z_2}$ for the action which exchanges $z_2$ and $z_3$.

\item The map $\tau_3\colon \fl(1,2,3,\infty)\to \fl(3,\infty)$ given by
$(p_1,p_2,p_3)\mapsto p_1+p_2+p_3$ pulls the Chern classes
$c_1,c_2,c_3$ of $\fl(3,\infty)$ back to the elementary symmetric polynomials
in $z_1,z_2,z_3$. The symmetric group $S_3$ acts on $\fl(1,2,3,\infty)$ by
permuting $p_1,p_2,p_3$, and the image of $\tau_3^*$ in rational cohomology is
exactly the fixed point subring $\Q[z_1,z_2,z_3]^{S_3}$.
\end{enumerate}
\end{exa}

\section{General theory}\label{Sec_5904_Gen}

In this section, we develop the general framework needed for our obstruction, which appears as \cref{cor:injectivity_Chern_class_criterion}.

\begin{dfn}\label{def:E_d,y}
Let $Y$ be a compact Hausdorff space, and let $d\in\N$. Fix a
separable infinite dimensional Hilbert space $H$ and set $K=K(H)$. We write
\[
E_{d,Y}=\{\omega \colon C(Y)\to K  \mid 
        \omega\text{ is a  homomorphism and }\rank(\omega(1))=d\}
\]
and, for any $m\geq d$, we set
\[
E_{d,Y,m}=\{\omega \colon C(Y)\to M_m  \mid 
        \omega\text{ is a homomorphism and }\rank(\omega(1))=d\}.
\]
We endow $E_{d,Y}$ and $E_{d,Y,m}$ with the topology of pointwise norm convergence, and view
$E_{d,Y,m}$ as embedded in $E_{d,Y}$. With this topology, the map
$\omega\mapsto \omega(1)$ defines a continuous map from $E_{d,Y}$ to the
Grassmannian of rank $d$ projections in $K$.

Recalling $M_{\infty}$ from Notation~\ref{ntn:M_infty}, we define
\[
\Ea_{d,Y}=\{\omega \colon C(Y)\to M_\infty  \mid 
        \omega\text{ is a homomorphism and }\rank(\omega(1))=d\}.
\]
The space $\Ea_{d,Y}$ is the increasing union of the spaces $E_{d,Y,m}$,
equipped with the direct-limit topology.
\end{dfn}

\begin{lem}\label{lem:E_alg_homotopy_equivalence}
	Let $Y$ be a compact Hausdorff space.
\begin{enumerate}
	\item\label{lem:E_alg_homotopy_equivalence_item_1} Let $Z$ be compact Hausdorff, let $A\subset Z$ be closed, and let
	$f\colon Z\to E_{d,Y}$ be continuous. Suppose that
	$f(A)\subset E_{d,Y,m_0}$ for some $m_0\geq d$. Then there exists
	$m\geq m_0$ such that $f$ is homotopic in $E_{d,Y}$, relative to $A$, to a map
	whose image is contained in $E_{d,Y,m}\subset\Ea_{d,Y}$.
	\item The inclusion $\iota \colon \Ea_{d,Y}\to E_{d,Y}$ is a weak homotopy
	equivalence.
\end{enumerate}
\end{lem}
\begin{proof}
For $z\in Z$, put $p_z=f(z)(1)$. Let $q_m$ be the projection onto the first
$m$ basis vectors of $H$. Since $Z$ is compact,
$q_m p_z q_m\to p_z$ uniformly in $z$. Choose $m\geq m_0$ such that
\[
\sup_{z\in Z}\|q_m p_zq_m-p_z\|<1/4.
\]
For $0\leq t\leq 1$, set
\[
b_{z,t}=(1-t)p_z+tq_m p_zq_m,
\]
and let
\[
r_{z,t}=\chi_{(1/2,\infty)}(b_{z,t})
\]
be the spectral projection of $b_{z,t}$. The projections $r_{z,t}$ are
continuous in $(z,t)$, satisfy $r_{z,0}=p_z$, and lie in $M_m$ when $t=1$.

Since $\|r_{z,t}-p_z\|<1$, the element $p_zr_{z,t}p_z$ is invertible in the
cutdown by $p_z$. Viewing $(p_z r_{z,t}p_z)^{-1/2}$ as an element in the
cutdown by $p_z$, the formula
\[
v_{z,t}=r_{z,t}p_z(p_z r_{z,t}p_z)^{-1/2}
\]
defines a continuous partial isometry with $v_{z,t}^*v_{z,t}=p_z$ and
$v_{z,t}v_{z,t}^*=r_{z,t}$. For $a\in C(Y)$, set
\[
f_t(z)(a)=v_{z,t} f(z)(a) v_{z,t}^* .
\]
Then $f_t(z)$ is a homomorphism $C(Y)\to K$ with $f_t(z)(1)=r_{z,t}$.
Thus $t\mapsto f_t$ is a homotopy in $E_{d,Y}$ from $f=f_0$ to a map
$f_1$ with image in $E_{d,Y,m}\subset\Ea_{d,Y}$.

It remains to check that the homotopy is relative to $A$. If $z\in A$, then
$p_z\in M_{m_0}\subset M_m$, so $q_m p_zq_m=p_z$. Hence
$b_{z,t}=p_z$, $r_{z,t}=p_z$, and $v_{z,t}=p_z$ for all $t$. Therefore
$f_t(z)=f(z)$ for all $z\in A$ and all $t$.

To show weak homotopy equivalence, we need to show that $\iota$ induces injective and surjective maps on all homotopy groups. The case of $\pi_0$ is immediate. Let $r \in \N$. Fix a basepoint $x_0 \in S^r$. Surjectivity follows from \eqref{lem:E_alg_homotopy_equivalence_item_1} with $Z = S^r$ and $A=\{x_0\}$. For injectivity, suppose $f_0,f_1 \colon S^r \to  \Ea_{d,Y}$ are two based continuous maps
such that $\iota \circ f_0$ and $\iota \circ f_1$ are homotopic.
Because $f_0 (S^r)$ and $f_1 (S^r)$ are compact,
there exists $m_0 \in \N$ such that these two images are contained in $E_{d,Y,m_0}$.
Set
\[
Z = S^r \times [0,1]
\andeqn
A=( S^r\times\{0,1\} ) \cup ( \{x_0\} \times [0,1] ).
\]
Then the homotopy is a based map $g \colon Z \to E_{d,Y}$ such that for $j=0,1$ we have $g|_{S^r \times \{j\}} = \iota \circ f_j$. Using \eqref{lem:E_alg_homotopy_equivalence_item_1}, there exists $m \geq m_0$ such that $g$ is homotopic relative to $A$ to a map whose image is in $E_{d,Y,m}$. The restriction of $\iota$ to $E_{d,Y,m}$ is a homeomorphism, so $f_0$ and $f_1$ are homotopic in $\Ea_{d,Y}$. 
\end{proof}

The following straightforward lemma plays an important role in the sequel. The proof is immediate, and is omitted.
\begin{lem} \label{lem_general_theory_h_vaprhi}
	Let $X$ and $Y$ be compact Hausdorff spaces, and let $d \in \N$. 
	For a homomorphism  $\varphi\colon C(Y)\to C(X,K)$  such that $\varphi(1)$ is a rank $d$ projection, define 
	$h_\varphi \colon X\to E_{d,Y}$ by 
	\[
	h_\varphi(x)(f)=\varphi(f)(x) .
	\] 
	The map $\varphi \mapsto h_{\varphi}$ defines a one-to-one correspondence between the space of such homomorphisms and the space of continuous maps from $X$ to $E_{d,Y}$.
	Furthermore, $\varphi_1$ is homotopic to $\varphi_2$ if and only if $h_{\varphi_1}$ is homotopic to $h_{\varphi_2}$. 
\end{lem}

By \cref{lem:E_alg_homotopy_equivalence}, applied with $Z=X$, the map 
$h_\varphi$ is
homotopic to a map whose image is contained in $\Ea_{d,Y}$. Replacing $\varphi$
by the homotopic homomorphism associated to this new map does not change the
induced map on $K$-theory. Thus, in the
homotopy invariance arguments below, we assume that
$h_\varphi\colon X\to \Ea_{d,Y}$.

We will need algebras of continuous functions on $E_{d, Y}$ or on
$\Ea_{d, Y}$, and their $K$-theory.
These spaces are not compact, so
we need pro-C*-algebras, or at least $\sigma$-C*-algebras.

Recall that a pro-C*-algebra is a complete Hausdorff topological
$*$-algebra whose topology is determined by its continuous
C*-seminorms
\cite[Definition~1.1]{PhillipsInverseLimits}
and that a $\sigma$-C*-algebra is a pro-C*-algebra
whose topology is determined by a countable family of such seminorms
\cite[introduction to Section~5]{PhillipsInverseLimits}.

If $Z$ is a topological space and $A$ is a C*-algebra,
we write $C (Z, A)$
for the $*$-algebra of continuous $A$-valued functions on $Z$, with the
topology of uniform convergence on compact subsets.
We write $C (Z)$ when $A=\C$.
If $Z$ is compactly generated, that is, if a subset $F\subset Z$ is closed if
and only if $F\cap K$ is closed for every compact subset $K\subset Z$, then
$C (Z, A)$ is a pro-C*-algebra.
See \cite[Example~1.3(3)]{PhillipsInverseLimits} for the case $A = \C$,
and see \cite[Lemma~2.4]{PhillipsInverseLimits} and the
preceding material in \cite[Section~2]{PhillipsInverseLimits} for the proof.
Although not stated in this generality, the proof applies
equally well to $C (Z, A)$.
In particular, this applies whenever $Z$ is metrizable.

Recall (see before \cite[Proposition~5.7]{PhillipsInverseLimits})
that a space $Z$ is countably compactly generated
if there is a countable family $\cC$ of compact subsets of $Z$ such
that a subset $F\subset Z$ is closed if and only if $F\cap K$ is closed for
every $K\in\cC$.

\begin{lem}\label{L_6706_CZA}
	Let $Z$ be countably compactly generated and let $A$ be a \ca.
	Then $C (Z, A)$ is a $\sigma$-C*-algebra.
\end{lem}

\begin{proof}
	Use \cite[Proposition~5.7 and Proposition~5.9(1)]{PhillipsInverseLimits},
	together with the identification,
	in \cite[Proposition~3.4]{PhillipsInverseLimits},
	of $C (Z, A)$ with $C (Z) \otimes A$.
\end{proof}

The main point
(in the proof of \cite[Proposition~5.7]{PhillipsInverseLimits}
when $A = \C$) is that
the seminorms $f \mapsto \sup_{z \in K}\ \|f (z) \|$, for $K \in \cC$,
define the topology of uniform convergence on compact subsets.

The advantage of using $\Ea_{d, Y}$ is the following lemma and its corollary.

\begin{lem}\label{lem_E_alg_countably_compactly_generated}
	Let $Y$ be compact Hausdorff and let $d \in \N$.
	Then $\Ea_{d, Y}$ is countably compactly generated.
\end{lem}

\begin{proof}
	For $m>d$, the stage $E_{d, Y, m}$ is the continuous image of the compact space
	$Y^d\times\fl(1,2,\ldots,d,m)$ under
	\[
	(y_1,\ldots,y_d;q_1,\ldots,q_d,q_{d+1})
	\longmapsto
	\left(f\mapsto \sum_{j=1}^d f(y_j)q_j\right),
	\]
	and, for $m=d$, the same formula using $Y^d\times\fl(1,2,\ldots,d)$ and the
	projections $q_1,\ldots,q_d$. In both cases, it follows that  
	$E_{d, Y, m}$ is compact.
	Since $\Ea_{d, Y}$ is equipped with the direct limit
	topology for the increasing union
	\[
	\Ea_{d, Y}=\bigcup_{m=d}^{\infty} E_{d, Y, m},
	\]
	a subset $F\subset \Ea_{d, Y}$ is closed if and only if
	$F\cap E_{d, Y, m}$ is closed for every $m\geq d$. Thus the countable family
	$\{E_{d, Y, m}:m\geq d\}$ of compact subsets shows that $\Ea_{d, Y}$ is
	countably compactly generated.
\end{proof}

\begin{cor}\label{cor_C_E_alg_sigma}
	Let $Y$ be compact Hausdorff, let $d \in \N$, and let $A$ be a C*-algebra.
	Then $C (\Ea_{d, Y}, A)$ is a $\sigma$-C*-algebra.
\end{cor}

\begin{proof}
	Combine \cref{L_6706_CZA} and \cref{lem_E_alg_countably_compactly_generated}.
\end{proof}

Kasparov's $KK$-theory has been extended to pro-C*-algebras
by Weidner \cite{WeidnerKKI, WeidnerKKII}.
Here we only need representable $K$-theory, and, by \cref{cor_C_E_alg_sigma},
only for $\sigma$-C*-algebras; this case is much easier.
For $\sigma$-C*-algebras, representable
$K$-theory was originally developed in \cite[Sections~1--3]{PhillipsRK}.
We recall the following equivalent but simpler
projection and unitary formulation of
\cite[Definition~1.1]{PhillipsToeplitz}.
Let $A$ be a $\sigma$-C*-algebra.
For any $\sigma$-C*-algebra~$B$, let $B^{+}$ denote its unitization.
Then $RK_1 (A)$ is the set of homotopy classes in the
set of unitaries $u \in (K \otimes A)^+$ such that $u - 1 \in K \otimes A$,
and $RK_0 (A)$ is the set of homotopy classes in the
set of projections $p \in M_2 \left ( (K \otimes A)^+ \right )$ such that
$p - \left( \begin{smallmatrix}
	1     &  0        \\
	0     &  0
\end{smallmatrix} \right) \in M_2 (K \otimes A)$.

We recall only the properties used in this paper.
The groups $RK_i (A)$ are functorial
homotopy invariant abelian groups, and for ordinary C*-algebras they agree
with ordinary $K$-theory
\cite[Propositions~1.4--1.7]{PhillipsToeplitz}.
They satisfy Bott periodicity
\cite[Theorem~2.4 and Theorem~3.7]{PhillipsToeplitz}, and
they satisfy the six-term exact sequence, countable additivity,
the Milnor $\lim^1$ sequence,
and stability \cite[Section~4.2]{PhillipsToeplitz}.
If $Z$ is countably
compactly generated, then $RK_*(C (Z))$ agrees with the usual representable
$K$-theory of $Z$ \cite[Theorem~3.3(1)]{PhillipsRK}; we denote this graded
group by $RK^*(Z)$.

We use stability in the following concrete form.
\begin{ntn}\label{ntn:kappa_Z}
	Let $Z$ be countably compactly generated, and let $e\in K$ be a rank one
	projection.
	Define
	\[
	\kappa_Z\colon C (Z) \to C (Z,K) 
	\; \text{ by } \;
	\kappa_Z(f)(z)=f(z)e .
	\]
\end{ntn}
The map $(\kappa_Z)_*$ is an isomorphism in representable $K$-theory, by
stability \cite[Section~4.2, Step~8, p.~248]{PhillipsToeplitz}. For compact
$Z$, this is the usual corner isomorphism $K_*(C(Z))\cong K_*(C (Z,K))$.

The arguments below use only functoriality, comparison, homotopy
invariance, Bott periodicity, exactness, and stability properties listed above.


\begin{dfn}\label{dfn_general_theory_def_of_psi}
Let $Y$ be a compact Hausdorff space, and let $d \in \N$. We define a canonical universal homomorphism 
\[
\psi_{Y,d} \colon C(Y)\to C(\Ea_{d,Y},K)
\]
by 
\[\psi_{Y,d}(f)(\eta)=\eta(f) .
\]
For a compact Hausdorff space $X$ and for $\varphi \colon C(Y) \to C(X,M_{\infty})$ such that $\varphi(1)$ has rank $d$, using $h_\varphi\colon X\to \Ea_{d,Y}$ as in Lemma~\ref{lem_general_theory_h_vaprhi}, define
\[
\lambda_\varphi^0\colon C(\Ea_{d,Y})\to C(X)
\; \text{ and } \;
\lambda_\varphi\colon C(\Ea_{d,Y},K)\to C(X,K)
\]
by
\[
\lambda_\varphi^0(f)(x)=f(h_\varphi(x)) 
\; \text{ and } \;
\lambda_\varphi(g)(x)=g(h_\varphi(x))
\]
for $f \in C(\Ea_{d,Y})$ and for $g \in C(\Ea_{d,Y},K)$. Note that $\lambda_\varphi = \lambda_\varphi^0 \otimes \id_{K}$. 
\end{dfn}
The algebra $C(\Ea_{d,Y},K)$ appears somewhat awkward, as the domain is the algebraic version of the space of homomorphisms while the codomain is the compact operators rather than $M_{\infty}$. However, these choices are needed so as to ensure that we get a $\sigma$-C*-algebra.
\begin{lem}\label{lem_eqn_general_theory_diagram}
	With the notation of Definition~\ref{dfn_general_theory_def_of_psi}, 
the following diagram commutes:
\[
\xymatrix{
C(Y) \ar[rr]^{\psi_{Y,d}} \ar[dr]_{\varphi} & & C(\Ea_{d,Y},K)
\ar[dl]^{\lambda_{\varphi}} \\
& C(X,K)
}
\]
\end{lem}
\begin{proof} 
	For any $f \in C(Y)$ and for any $x \in X$, we have 
	\[
	((\lambda_\varphi\circ \psi_{Y,d})(f))(x)
	=\psi_{Y,d}(f)(h_\varphi(x))=h_\varphi(x)(f)=\varphi(f)(x) .
	\]
	This completes the proof.
\end{proof}

\begin{lem}\label{lem:universal-map-injective}
After identifying $K_0(C(Y))$ with $RK_0(C(Y))$, the map $\psi_{Y,d}$ from
Definition~\ref{dfn_general_theory_def_of_psi} induces an injective map
\[
        (\psi_{Y,d})_*\colon K_0(C(Y))\to RK_0(C(\Ea_{d,Y},K)).
\]
\end{lem}
\begin{proof}
In Lemma~\ref{lem_eqn_general_theory_diagram}, take $X=Y$. Fix pairwise orthogonal
rank one projections
$e_1,e_2,\ldots,e_d\in M_d\subset M_\infty\subset K$, fix a point
$y_0\in Y$, and define
\[
        \varphi(f)(y)=f(y)e_1+f(y_0)(e_2+\cdots+e_d)
        \; \text{ for }  f\in C(Y) \text{ and } y\in Y.
\]
This is a homomorphism from the C*-algebra $C(Y)$ to the C*-algebra 
$C(Y,K)$. Hence its induced map in representable $K$-theory agrees
with the usual map on $K$-theory. 

With the usual identification of $K_0(C(Y))$ with $K_0(C(Y,K))$, we have
\[
        \varphi_*(\xi)=\xi+(d-1)(\ev_{y_0})_*(\xi)[1_Y]
        \; \text{ for } \; \xi\in K_0(C(Y)),
\]
and in particular, $\varphi_*$ is injective. 

By functoriality of representable $K$-theory and the commuting diagram in
Lemma~\ref{lem_eqn_general_theory_diagram}, we have $\varphi_*=(\lambda_\varphi)_*\circ (\psi_{Y,d})_*$.
Since $\varphi_*$ is injective, $(\psi_{Y,d})_*$ is injective.
\end{proof}

In the next lemma, as above, we
 identify $RK^*(\Ea_{d,Y})$ with $RK_*(C(\Ea_{d,Y}))$, using  
\cite[Theorem~3.3(1)]{PhillipsRK}. 

\begin{lem} \label{lem_pullback_identification}
	Let $\varphi \colon C(Y) \to C(X,M_{\infty})$ be a homomorphism such that $\rank(\varphi(1)) = d$. 
	With the notation as in Notation~\ref{ntn:kappa_Z} and \cref{dfn_general_theory_def_of_psi},
	if $\xi \in K_0(C(Y))$ and
	\[
	\zeta=\left (\kappa_{\Ea_{d,Y}} \right )_*^{-1} \left (  \left ( \psi_{Y,d} \right )_*(\xi) \right ) \in RK^0(\Ea_{d,Y}) ,
	\]
	then
	$ h_\varphi^*(\zeta)=\left ( (\kappa_X )_*^{-1}\circ \varphi_* \right ) (\xi)$.
\end{lem}
To simplify notation, when applying the lemma, we will suppress the corner identification and simply write this
as $h_\varphi^*(\zeta)=\varphi_*(\xi)$.
\begin{proof}[Proof of \cref{lem_pullback_identification}]
We have 
$ \lambda_\varphi\circ \kappa_{\Ea_{d,Y}}=\kappa_X\circ \lambda_{\varphi}^0 $
on scalar-valued functions. Naturality of $RK_*$ therefore gives the
commutative diagram
\[
\xymatrix@C=3.5em{
        RK^*(\Ea_{d,Y}) \ar[r]^-{\left ( \kappa_{\Ea_{d,Y}} \right )_*}_{\cong}
        \ar[d]_{h_\varphi^*}
        & RK_*(C(\Ea_{d,Y},K)) \ar[d]^{(\lambda_\varphi)_*} \\
        K^*(X) \ar[r]^-{(\kappa_X)_*}_{\cong} & K_*(C(X,K)).
}
\]
The conclusion now follows from Lemma~\ref{lem_eqn_general_theory_diagram}.
\end{proof}

We now specialize to $Y=S^{2n}$ and $X=S^{2k}$. Let
\[
        \beta\in \widetilde K^0(S^{2n})
\]
be the reduced Bott class. It has rank zero, and   
$K^0(S^{2n})\cong \Z[1]\oplus\Z\beta$. Define $\eta,\xi \in RK^0(\Ea_{d,S^{2n}})$ by
\[
	\eta
	=\left ( \kappa_{\Ea_{d,S^{2n}}} \right )_*^{-1}
	\left ( \psi_{S^{2n},d} \right )_*([1])
	\;\;\; \text{ and } \;\;\;
	\xi =\left ( \kappa_{\Ea_{d,S^{2n}}} \right )_*^{-1}
	\left ( \psi_{S^{2n},d} \right )_*(\beta)
	.
\]
Since $(\psi_{S^{2n},d})_*$ is injective by
\cref{lem:universal-map-injective}, the classes $\eta$ and $\xi$ span the image
of $K^0(S^{2n})$ in $RK^0(\Ea_{d,S^{2n}})$. 
For a homomorphism $\varphi \colon C(S^{2n}) \to C(S^{2k},K)$ such that $\rank(\varphi(1)) = d$, 
under the identification from Lemma~\ref{lem_pullback_identification}, we have, in $K^0(S^{2k})$,
\[
        h_\varphi^*(\eta)=\varphi_*([1])
        \;\;\; \text{ and } \;\;\;
        h_\varphi^*(\xi)=\varphi_*(\beta) .
\]
The class $\xi$ has rank zero.

The Chern character $\Ch(\xi)$ used below is understood stagewise: on every
compact stage $E_{d,S^{2n},m}$ it is the ordinary rational Chern character of
the restriction of $\xi$, and these classes are compatible under restriction.
Equivalently, for any compact Hausdorff space $T$ and for any continuous map
$g\colon T\to \Ea_{d,S^{2n}}$, one has
\[
        g^* ( \Ch(\xi) )=\Ch(g^* ( \xi ) ).
\]
Thus, when we write $h_\varphi^*( \Ch(\xi) )$, this means
$\Ch(h_\varphi^* ( \xi ) )$ after first pulling $\xi$ back to  
$S^{2k}$. 

\begin{cor}\label{cor:injectivity_Chern_class_criterion}
With the notation above, if the degree $2k$ component of $\Ch(\xi)$ is
decomposable in $\vH^*(\Ea_{d,S^{2n}};\Q)$, then $\varphi_*$ is not rationally
injective on $K_0$.
\end{cor}
 In particular, $\varphi_*$ is not injective on $K_0$.
\begin{proof}[Proof of \cref{cor:injectivity_Chern_class_criterion}]
	On $S^{2k}$, the ordinary rational Chern
	character identifies $K^0(S^{2k})\otimes\Q$ with
	\[
	\vH^0(S^{2k};\Q)\oplus \vH^{2k}(S^{2k};\Q).
	\]
	The class $\varphi_*([1])$ has nonzero rank $d$, while $\varphi_*(\beta)$ has
	rank zero. Therefore $\varphi_*\otimes \id_{\Q}$ is injective on
	$K_0(C(S^{2n}))\otimes\Q$ if and only if $\varphi_*(\beta)$ is nonzero;
	equivalently, if and only if 
	\[
	h_\varphi^* ( \Ch(\xi) ) \neq 0
	\quad\text{in }\widetilde{\vH}^{2k}(S^{2k};\Q).
	\]
	
Any positive-degree class on $S^{2k}$ has degree $2k$, so every decomposable
class of degree $2k$ pulls back to zero in $\vH^{2k}(S^{2k};\Q)$. If the degree
$2k$ component of $\Ch(\xi)$ is decomposable, then
$h_\varphi^* ( \Ch(\xi) ) =0$. Therefore,
$\varphi_*\otimes \id_{\Q}$ is not injective on $K_0$, and hence $\varphi_*$ is not
injective on $K_0$.
\end{proof}

\section{The main theorem}\label{Sec_5904_General_rank}
\indent
In this section we settle Blackadar's problem, in \cref{thm:rank-d-obstruction}. Its proof involves a computation of the cohomology ring of the homomorphism space $E_{d,Y}$; part of that is done for general spaces $Y$, and for the remaining parts, we specialize to the case of an even sphere. 

Throughout this section, $n\in\N$ is fixed and the domain is the sphere
$S^{2n}$.  The letter $d\in\N$ denotes the rank of the image of the
identity.
Thus a homomorphism
\[
        \varphi\colon C(S^{2n})\longrightarrow C(S^{2k},K)
\]
has rank $d$ if $\varphi(1)$ is a projection of constant rank $d$.

We say that a space $X$ is \emph{rationally acyclic} if 
$\widetilde{\vH}^*(X;\Q)=0$.

The result proved below is the following.

\begin{thm}\label{thm:rank-d-obstruction}
Let $n,d,k\in\N$.  If
\[
        k\notin\{n,n+1,\ldots,n+d-1\},
\]
then:
\begin{enumerate}
	\item\label{thm:rank-d-obstruction_item_1}  There is no homomorphism
	\[
	\varphi\colon C(S^{2n})\longrightarrow C(S^{2k},K)
	\]
	such that $\varphi(1)$ has rank $d$ and $\varphi_*$ is injective on
	$K_0$.  
	\item\label{thm:rank-d-obstruction_item_2} For any $r \in \N$, there is no unital homomorphism
	\[
	C(S^{2n} , M_r) \longrightarrow C(S^{2k},M_{dr})
	\]
	which is injective on $K_0$.
\end{enumerate}
\end{thm}
Before getting to the proof, we need to compute the cohomology ring of the appropriate space of homomorphisms. This is done in several steps.
\subsection{The rank $d$ parameter space}

Let $X$ be a compact Hausdorff space.  Recall from Definition~\ref{def:E_d,y} that
\begin{equation}
   E_{d,X}=\{\omega \colon C(X)\to K \mid
\omega \hbox{ is a homomorphism and }
\rank(\omega(1))=d\},
\end{equation}
with the topology of pointwise norm convergence.  Also recall the algebraic model
\begin{equation}
   \Ea_{d,X}=\{\omega \colon C(X)\to M_\infty \mid
        \omega \hbox{ is a homomorphism and }
        \rank(\omega(1))=d\},
\end{equation}
with the direct limit topology.  By \cref{lem:E_alg_homotopy_equivalence},
the inclusion $\Ea_{d,X}\to E_{d,X}$ is a weak homotopy equivalence.  The
universal $K$-theory construction used below is the one from
Definition~\ref{dfn_general_theory_def_of_psi}; the explicit cohomology calculation is
made on the ordered flag model described next.

\begin{ntn}
	Write for short
\[
   \Fl_d^\infty=\fl(1,2,\ldots,d,\infty) \;\;\; \text{ and }\;\;\; \Fl_d^{\infty,\alg}=\fl^{\alg}(1,2,\ldots,d,\infty) .
\]
\end{ntn}
  Recall from \cref{thm:stable-flag-cohomology}
that their rational cohomology is the polynomial algebra
\[
   \vH^*(\Fl_d^\infty;\Q)\cong \Q[\zeta_1,\ldots,\zeta_d],
        \qquad \deg(\zeta_i)=2,
\]
where $\zeta_i$ is the first Chern class of the $i$-th tautological line
bundle.  The weak equivalence between the compact-operator and
algebraic versions of the flag space is \cref{lem:alg_flag_manifold}.

\begin{ntn}
Set
\[
   P_d(X)=X^d\times \Fl_d^\infty \;\;\; \text{ and } \;\;\;
      P_d^{\mathrm{alg}}(X)=X^d\times \Fl_d^{\infty,\alg} .
\]
There is a natural quotient map
\[
   \pi\colon P_d(X)\longrightarrow E_{d,X}
\]
defined, for $f\in C(X)$, by
\begin{equation*}
   \pi(x_1,\ldots,x_d;q_1,\ldots,q_d)(f)
      =\sum_{j=1}^d f(x_j)q_j .
\end{equation*}
There is also the algebraic version
\[
   \pi^{\mathrm{alg}}\colon P_d^{\mathrm{alg}}(X)\longrightarrow \Ea_{d,X},
\]
given by the same formula.
The symmetric group $S_d$ acts on $P_d(X)$ and $P_d^{\mathrm{alg}}(X)$ by
simultaneous permutation of the points and projections: for $\sigma \in S_d$,
\[
    \quad \sigma\cdot (x_1,\ldots,x_d;q_1,\ldots,q_d)
    = (x_{\sigma^{-1}(1)},\ldots,x_{\sigma^{-1}(d)};
       q_{\sigma^{-1}(1)},\ldots,q_{\sigma^{-1}(d)}) .
\]
\end{ntn}
\begin{lem}
	The maps $\pi$ and $\pi^{\mathrm{alg}}$  descend to surjective maps
\[
   \bar{\pi}\colon P_d(X)/S_d\longrightarrow E_{d,X}
   \quad\text{and}\quad
   \bar{\pi}^{\mathrm{alg}}\colon P_d^{\mathrm{alg}}(X)/S_d\longrightarrow \Ea_{d,X}.
\]
\end{lem}
\begin{proof}
	The maps $\pi$ and $\pi^{\mathrm{alg}}$ are constant on $S_d$-orbits, and therefore induce well-defined maps on the  quotient spaces.
We now show that these maps are surjective. If
$\omega\in E_{d,X}$, then, since $C(X)$ is commutative and $\omega(1)$ has
finite rank, $\omega$ is a finite-dimensional representation of $C(X)$.  Hence
there are distinct points $y_1,y_2,\ldots,y_l\in X$ and pairwise orthogonal
projections $p_1,p_2,\ldots,p_l$ such that
\begin{equation}\label{eq:omega-def}
   \omega(f)=\sum_{r=1}^l f(y_r)p_r 
     \; \text{ for } \; f\in C(X).
\end{equation}
Writing $s_r=\rank(p_r)$, we have $s_1+\cdots+s_l=d$.  Choosing an ordered
orthogonal decomposition of each $p_r$ into $s_r$ rank one projections gives a
point in $P_d(X)$ which maps to $\omega$.

If $\omega \in \Ea_{d,X}$, then $\omega(1) \in M_{\infty}$. Therefore there exists $m \in \N$ such that $\omega(1) \in M_m$, so we get a point of $P_d^{\mathrm{alg}}(X)$.
\end{proof}

\begin{dfn}\label{dfn:collision-type}
Let $\omega\in E_{d,X}$, and with the notation as in \cref{eq:omega-def}, write
\[
   \omega(f)=\sum_{r=1}^l f(y_r)p_r.
\]
Set $s_r=\rank(p_r)$ for $r=1,2,\ldots,l$. If the projections are ordered so that $s_1 \geq s_2 \geq \cdots \geq s_{l}$, the tuple
$ (s_1,\ldots,s_l) $
is called the \emph{collision type} of $\omega$.  
\end{dfn}

\subsection{Flag manifolds appearing in the fibers}

We recall notation from \cref{Sec_5904_Prelim}.
We use two different kinds of flag manifolds.  The space $\Fl_d^\infty$ above 
is infinite
dimensional; by \cref{thm:stable-flag-cohomology}, its rational cohomology is a
polynomial algebra, with no elementary-symmetric-polynomial relations. We also
need finite-dimensional complete flag manifolds. For $s\in\N$, write
\[
       \Fl_s=\fl(1,2,\ldots,s)=U(s)/\T^s.
\]
We use the notation from \cref{prp:finite-complete-flag-cohomology}: the
classes $z_1,\ldots,z_s\in H^2(\Fl_s;\Q)$ are the first Chern classes of
the tautological line bundles, and $e_j$ denotes the $j$-th elementary symmetric
polynomial in $z_1,\ldots,z_s$. Thus the finite flag relations are
$e_1=\cdots=e_s=0$.  

\begin{exa}
	We work out the details for what happens in the rank four case with 
	collision type $(2,1,1)$.

Let $d=4$, let $x,y,z\in X$ be distinct, and let
$p_x,p_y,p_z \in K$ be pairwise orthogonal projections with $\rank(p_x)=2, \rank(p_y)=1$ and $\rank(p_z)=1$.
Let
\[
\omega(f)=f(x)p_x+f(y)p_y+f(z)p_z.
\]
This point of $E_{4,X}$ has collision type $(2,1,1)$.

Consider the ordered support tuple
\[
t_0=(x,x,y,z).
\]
The stabilizer of $t_0$ in $S_4$ is
\[
\Stab_{S_4}(t_0)=\{1,(12)\}\cong S_2.
\]
The other elements of $S_4$ do not preserve this ordered support tuple.  They
carry it to one of the other ordered tuples in the orbit of $(x,x,y,z)$, such as
\[
(x,y,x,z),\qquad (x,y,z,x),\qquad (y,x,x,z),
\]
and so on.  The orbit has
\[
\frac{4!}{2!1!1!}=12
\]
elements.

For each ordered tuple $t=(t_1,t_2,t_3,t_4)$ in this orbit, define
\[
F_t=\{(x_1,x_2,x_3,x_4;q_1,q_2,q_3,q_4)
\in \pi^{-1}(\omega) \; \mid \; (x_1,x_2,x_3,x_4) = (t_1,t_2,t_3,t_4) \}.
\]
The allowed choices of $(q_1,q_2,q_3,q_4)$ are $4$-tuples of orthogonal rank $1$ projections satisfying
\[
\sum_{\{j \mid t_j=x\}}q_j=p_x,\qquad
\sum_{\{j \mid t_j=y\}}q_j=p_y,
\; \; \text{ and } \; \;
\sum_{\{j \mid t_j=z\}}q_j=p_z.
\]
Then
\[
\pi^{-1}(\omega)=\coprod_{t\in S_4\cdot t_0}F_t.
\]
For the chosen tuple $t_0=(x,x,y,z)$, these conditions are
\[
q_1+q_2=p_x,
\qquad q_3=p_y,
\qquad q_4=p_z.
\]
The only freedom is the ordered
orthogonal decomposition of the rank $2$ projection $p_x$.  Hence
\[
F_{t_0}\cong \Fl_2 \cong \mathbb{C}P^1.
\]
The stabilizer $S_2=\{1,(12)\}$ acts on $F_{t_0}$ by switching the two 
projections:
\[
(q_1,q_2,q_3,q_4)\longmapsto(q_2,q_1,q_3,q_4).
\]
Every element of $S_4$ not in this stabilizer carries $F_{t_0}$ to a different
component $F_t$.  Therefore the quotient of the full ordered fiber by $S_4$ is
\begin{equation}\label{eq:rank4-211-quotient}
\bar{\pi}^{-1}(\omega)
=\pi^{-1}(\omega)/S_4
\cong F_{t_0}/\Stab_{S_4}(t_0)
\cong \Fl_2 /S_2.
\end{equation}
In this example $\Fl_2 /S_2\cong \mathbb RP^2$, so it has rational
cohomology $\Q$ in degree zero and no positive degree rational cohomology.  The
ordered fiber, by contrast, is a disjoint union of $12$ copies of
$\mathbb{C}P^1$.  Thus the ordered fiber is very far from rationally acyclic,
while the quotient fiber is rationally acyclic.
\end{exa}

\begin{lem}\label{lem:general-quotient-fiber}
Suppose $\omega\in E_{d,X}$ has collision type $(s_1,\ldots,s_l)$, that is, $	s_1+\cdots+s_l=d$,
and there are distinct points $y_1,y_2,\ldots,y_l$ and pairwise orthogonal projections $p_1,p_2,\ldots,p_l$ such that $\rank(p_r) = s_r$ for $r=1,2,\ldots,l$ and 
\begin{equation}\label{eq:eta-decomposition}
	\omega(f)=\sum_{r=1}^l f(y_r)p_r .
\end{equation}
Then
\begin{equation}\label{eq:general-quotient-fiber}
	\bar{\pi}^{-1}(\omega)
	\cong
	\prod_{r=1}^l ( \Fl_{s_r} /S_{s_r} ) .
\end{equation}
\end{lem}
\begin{proof}
Let
$\rho\colon \{1,\ldots,d\} \to  \{1,\ldots,l\}$
be a function such that
\[
\card(\rho^{-1}(r))=s_r
\quad \text{ for } r=1,\ldots,l.
\]
The function $\rho$ records which ordered coordinates are assigned to the
point $y_r$.
Define
\begin{align*}
	F_\rho
	=\Bigl \{ &(y_{\rho(1)},\ldots,y_{\rho(d)};q_1,\ldots,q_d)
	\in P_d(X) \mid                                  \\
	&\hspace{1cm}
	\sum_{j\in\rho^{-1}(r)}q_j=p_r
	\hbox{ for } r=1,\ldots,l \Bigr \}.
\end{align*}
Then
\begin{equation}\label{eq:ordered-fiber-disjoint-union}
   \pi^{-1}(\omega)=\coprod_{\rho} F_\rho,
\end{equation}
where $\rho$ ranges over all functions with the displayed cardinality
conditions.  There are
\[
\dfrac{d!}{s_1!s_2!\cdots s_l!}
\]
such functions.

For a fixed $\rho$, the equation
$   \sum_{j\in\rho^{-1}(r)}q_j=p_r $
says that the rank $s_r$ projection $p_r$ is decomposed into $s_r$ rank $1$ 
projections, indexed by the set $\rho^{-1}(r)$, listed in a given order.  
Therefore
\begin{equation}\label{eq:F-rho-product-flags}
       F_\rho\cong \prod_{r=1}^l \Fl_{s_r}.
\end{equation}

Choose one such function $\rho_0$.  Its stabilizer in $S_d$ is the 
subgroup
\begin{equation}\label{eq:young-stabilizer}
   H_{\rho_0}=\Stab_{S_d}(\rho_0)
      =\prod_{r=1}^l S_{\rho_0^{-1}(r)}
      \cong S_{s_1}\times\cdots\times S_{s_l}.
\end{equation}
The group $S_d$ acts transitively on the components $F_\rho$ in
\eqref{eq:ordered-fiber-disjoint-union}.  Therefore, by
\cref{lem:transitive-components},
\begin{equation}\label{eq:quotient-fiber-one-component}
   \bar{\pi}^{-1}(\omega)
       =\pi^{-1}(\omega)/S_d
       \cong F_{\rho_0}/H_{\rho_0}.
\end{equation}
Under the identification \eqref{eq:F-rho-product-flags}, the group
$H_{\rho_0}\cong S_{s_1}\times\cdots\times S_{s_l}$ acts factor by factor.
\end{proof}

\begin{rmk}\label{rmk:no-extra-symmetric-groups}
There is no extra symmetric group permuting two distinct collision blocks of
the same size.  For example, in collision type $(2,2)$, the two rank two
blocks lie over two distinct points of $X$.  The stabilizer of one ordered
component preserves the block over the first point and the block over the
second point separately, so the stabilizer is $S_2\times S_2$, not
$(S_2\times S_2)\rtimes S_2$.
\end{rmk}

\subsection{Rational acyclicity of the unordered fibers}

\begin{lem}\label{lem:block-quotient-rational-acyclic}
For every $s\in\N$, the quotient $\Fl_s /S_s$ is rationally acyclic; that is,
$\widetilde{\vH}^*(\Fl_s /S_s;\Q)=0$.
\end{lem}

\begin{proof}
By \cref{thm:cohomology-quotient}, we have 
\begin{equation}\label{eq:transfer-fixed-points}
       \vH^*(\Fl_s /S_s;\Q)
       \cong \vH^*(\Fl_s;\Q)^{S_s}.
\end{equation}
The group $S_s$ acts on $\Fl_s$ by permuting the $s$ rank one
projections, and therefore acts on \eqref{eq:finite-flag-cohomology} by permuting
$z_1,\ldots,z_s$.

With $e_1,e_2,\ldots,e_s$ denoting the first $s$ elementary 
symmetric polynomials, let
$ R=\Q[z_1,\ldots,z_s]$,       
and let $I$ be the ideal generated by $e_1,e_2,\ldots,e_s$.
Then $\vH^*(\Fl_s ;\Q)\cong R/I$.  Let $f\in R$ and assume that $f+I$ is fixed by
$S_s$.  Average $f$ over $S_s$:
\[
       \Avg(f)=\frac{1}{s!}\sum_{\sigma\in S_s}\sigma f.
\]
Since $f+I$ is fixed in $R/I$, we have $f+I=\Avg(f) + I$.  The polynomial
$\Avg(f)$ is symmetric, so $ \Avg(f)\in R^{S_s}=\Q[e_1,\ldots,e_s]$.
Every positive degree symmetric polynomial is in $I$. Thus $\Avg(f)$ is the sum of a constant term and an element of $I$, so the image of $f$ in $R/I$ is in the degree $0$ part.
\end{proof}

\begin{cor}\label{cor:all-fibers-rational-acyclic}
For every $\omega\in E_{d,X}$, the fiber $\bar{\pi}^{-1}(\omega)$ is rationally 
acyclic.
\end{cor}

\begin{proof}
Let $(s_1,\ldots,s_l)$ be the collision type of $\omega$.  By
Lemma \ref{lem:general-quotient-fiber},
\[
       \bar{\pi}^{-1}(\omega)
       \cong
       \prod_{r=1}^l\bigl(\Fl_{s_r} /S_{s_r}\bigr).
\]
By \cref{lem:block-quotient-rational-acyclic}, each factor has rational
\v{C}ech cohomology $\Q$ in degree zero, and $0$ in other degrees. The result
now follows from the K\"unneth theorem. 
\end{proof}

We recall the following form of the Vietoris--Begle theorem.

\begin{thm}[Vietoris--Begle; see {\cite[Chapter 6, Section 9, Theorem 15, p.~344]{Spanier}}] \label{thm:vietoris-begle}
Let $f\colon Z\to Y$ be a closed surjective map of paracompact Hausdorff spaces.
Suppose that, for every $y\in Y$,
\[
       \widetilde{\vH}^*(f^{-1}(y);\Q)=0.
\]
Then
\[
       f^*\colon \vH^*(Y;\Q)\longrightarrow \vH^*(Z;\Q)
\]
is an isomorphism.
\end{thm}

\begin{lem}\label{lem:algebraic-quotient-map-closed}
For compact Hausdorff $X$, the map
\[
\bar\pi^{\mathrm{alg}}\colon
P_d^{\mathrm{alg}}(X)/S_d
\longrightarrow
\Ea_{d,X}
\]
is closed.
\end{lem}
\begin{proof}
Let $P_{d,m}(X)$ be the space of $d$-tuples of points of $X$, together with
$d$ mutually orthogonal rank one projections in $M_m$, and let $E_{d,X,m}$ be
the space of homomorphisms $C(X)\to M_m$ for which the image of $1$ has rank
$d$. Recall that
\[
P_d^{\mathrm{alg}}(X)/S_d=\bigcup_{m\geq d} P_{d,m}(X)/S_d
\]
with the direct limit topology, and
\[
\Ea_{d,X}=\bigcup_{m\geq d} E_{d,X,m}
\]
with the direct limit topology.

Let $Z\subset P_d^{\mathrm{alg}}(X)/S_d$ be closed. Since $\Ea_{d,X}$ has the
direct limit topology, it is enough to prove that
$\bar\pi^{\mathrm{alg}}(Z)\cap E_{d,X,m}$ is closed in $E_{d,X,m}$ for every
$m\geq d$.

We first claim that
\[
   (\bar\pi^{\mathrm{alg}})^{-1}(E_{d,X,m})=P_{d,m}(X)/S_d.
\]
Indeed, suppose that a point of $P_d^{\mathrm{alg}}(X)/S_d$ is represented by
$(x_1,\ldots,x_d;q_1,\ldots,q_d)$ and that its image under
$\bar\pi^{\mathrm{alg}}$ belongs to $E_{d,X,m}$. Then
\[
        \sum_{i=1}^d q_i
        =\bar\pi^{\mathrm{alg}}(x_1,\ldots,x_d;q_1,\ldots,q_d)(1)
\]
is a projection in $M_m$. Since $q_1,q_2,\ldots,q_d$ are mutually orthogonal rank one
projections and their sum lies in $M_m$, for each $j$, $q_j$ is a subprojection of a
projection in $M_m$. Hence $q_j\in M_m$. The reverse inclusion is
immediate. The claim is proved.

Let
\[
\bar\pi_m\colon P_{d,m}(X)/S_d\longrightarrow E_{d,X,m}
\]
be the finite stage quotient map. Then
\[
\bar\pi^{\mathrm{alg}}(Z)\cap E_{d,X,m}
=
\bar\pi_m\bigl(Z\cap(P_{d,m}(X)/S_d)\bigr).
\]
Since $Z$ is closed in the direct limit topology,
$Z\cap(P_{d,m}(X)/S_d)$ is closed in $P_{d,m}(X)/S_d$. The space
$P_{d,m}(X)/S_d$ is compact Hausdorff, and $E_{d,X,m}$ is Hausdorff. Therefore
$\bar\pi_m$ is closed, so the displayed image is closed in $E_{d,X,m}$.
This proves the lemma.
\end{proof}

\begin{prp}\label{prp:cohomology-parameter-space-fixed-ring}
For compact metrizable $X$, the pullback map
\[
   (\bar\pi^{\mathrm{alg}})^*\colon
   \vH^*(\Ea_{d,X};\Q)
   \longrightarrow
   \vH^*(P_d^{\mathrm{alg}}(X)/S_d;\Q)
\]
is an isomorphism. Moreover,
\[
   \vH^*(\Ea_{d,X};\Q)
   \cong
   \vH^*(P_d^{\mathrm{alg}}(X);\Q)^{S_d}.
\]
\end{prp}
\begin{proof}
The spaces $P_d^{\mathrm{alg}}(X)/S_d$ and $\Ea_{d,X}$ are countable direct
limits of compact metrizable spaces along closed inclusions.  In particular,
they are paracompact Hausdorff. The map $\bar\pi^{\mathrm{alg}}$ is surjective and is
closed by \cref{lem:algebraic-quotient-map-closed}. Its fibers are the 
finite flag quotients from Lemma~\ref{lem:general-quotient-fiber}; by
\cref{cor:all-fibers-rational-acyclic}, they are rationally acyclic. Therefore
\cref{thm:vietoris-begle} gives the first isomorphism.

The second isomorphism follows from the first isomorphism and
\cref{thm:cohomology-quotient}, applied to the finite group action of $S_d$ on
$P_d^{\mathrm{alg}}(X)$.
\end{proof}

\subsection{The fixed points of the cohomology ring of the ordered model}

Now take $X=S^{2n}$.  Let
\[
       a_j\in \vH^{2n}((S^{2n})^d;\Q)
\]
be the pullback of a fixed generator of $\widetilde{\vH}^{2n}(S^{2n};\Q)$ from
the $j$-th sphere factor.  Then
$a_j^2=0$ and $a_j a_l \neq 0$ for $j\neq l$.

Let $L_j$ be the $j$-th tautological line bundle over the stable flag factor
$\Fl_d^{\infty,\text{alg}}$, and set
\[
       \zeta_j=c_1(L_j).
\]
By the K\"unneth theorem and \cref{thm:stable-flag-cohomology},
\begin{equation}\label{eq:Pd-cohomology-ring}
       \vH^*(P_d^{\mathrm{alg}}(S^{2n});\Q)
       \cong
       \Q[\zeta_1,\ldots,\zeta_d]
       \otimes
       \Q[a_1,\ldots,a_d]/(a_1^2,\ldots,a_d^2).
\end{equation}
The group $S_d$ acts by simultaneously permuting the pairs
$(a_i,\zeta_i)$.

Set
\begin{equation}\label{eq:varepsilon-def}
       \varepsilon_j=e_j(\zeta_1,\ldots,\zeta_d) \;
       \; \text{ for } \; j=1,\ldots,d,
\end{equation}
and, for $m\in\Nz$, set
\begin{equation}\label{eq:nu-m-def}
       \nu_m=\sum_{j=1}^d a_j\zeta_j^m.
\end{equation}
Thus $\deg(\varepsilon_j)=2j$ and $\deg(\nu_m)=2n+2m$.

\begin{lem}\label{lem:one-a-module}
Let
\[
       T=\Q[\varepsilon_1,\ldots,\varepsilon_d]
       =\Q[\zeta_1,\ldots,\zeta_d]^{S_d}.
\]
The $S_d$-fixed part of the subspace of 
$\vH^*(P_d^{\mathrm{alg}}(S^{2n});\Q)$ given by 
\[
 \sum_{j=1}^d \Q[\zeta_1,\ldots,\zeta_d] \cdot a_j 
\]
 is a free
$T$-module with basis
$ \nu_0,\nu_1,\ldots,\nu_{d-1}$ as defined in \eqref{eq:nu-m-def}.
\end{lem}

\begin{proof}
		Let $R=\Q[\zeta_1,\ldots,\zeta_d]$.
		Let $S_{\{2,\ldots,d\}}$ be the subgroup of $S_d$ that fixes $1$.
		
	We first describe the structure of $R^{S_{\{2,\ldots,d\}}}$ as a \(T\)-module.  
		Setting 
	\[
	\varepsilon_j^{(1)}
	=
	e_j(\zeta_2,\ldots,\zeta_d),
	\]
	we have
	\[
	R^{S_{\{2,\ldots,d\}}}
	=
	\Q[\zeta_1,\varepsilon_1^{(1)},\ldots,\varepsilon_{d-1}^{(1)}].
	\]	
	Set
	\[
	p(t)
	=
	\prod_{j=1}^d(t-\zeta_j) =
	t^d-\varepsilon_1t^{d-1}
	+\varepsilon_2t^{d-2}
	-\cdots
	+(-1)^d\varepsilon_d
	\in T[t].
	\]
	Define a  \(T\)-algebra homomorphism
	$\Phi \colon T[t] \to R^{S_{\{2,\ldots,d\}}}$
	by $\Phi(t) = \zeta_1$. Because $p(\zeta_1) = 0$, the homomorphism
	$\Phi$ vanishes on the ideal generated by $p(t)$; denote this ideal by $J$,
	and write 
	\[
	\bar{\Phi} \colon T[t]/J \to R^{S_{\{2,\ldots,d\}}}
	\] 
	for the induced map on the quotient.
	
	We claim that the map $\bar{\Phi}$ is an isomorphism.  To see this, with
	$\varepsilon_j$ as in \eqref{eq:varepsilon-def},
	define polynomials \(g_j(t)\in T[t]\) recursively by
	$g_0(t)=1$
	and $g_j(t)=\varepsilon_j-tg_{j-1}(t)$
	for $j=1,2,\ldots,d-1$.
 	Then
	$\Phi(t) = \zeta_1$ and one checks by induction that $\Phi(g_j) = \varepsilon_j^{(1)}$ for $j=1,2,\ldots,d-1$. Hence $\Phi$ is surjective. It remains to show that the kernel of $\Phi$ is contained in $J$. Indeed, suppose
	$q(t) \in T[t]$ satisfies $\Phi(q) = 0$. Then $q(\zeta_1) = 0$, so $q(t)$ is divisible by $t-\zeta_1$. Because $q$ is fixed by the action of $S_d$, it must be divisible by $\prod_{j=1}^d(t-\zeta_j) = p(t)$, so $q \in J$, as claimed.

	Since \(p(t)\) is monic of degree \(d\), the quotient \(T[t]/(p(t))\) is a
	free \(T\)-module with basis
	$ 1,t,\ldots,t^{d-1} $.
	Equivalently,
	\[
	R^{S_{\{2,\ldots,d\}}}
	\cong
	T\oplus T \zeta_1\oplus\cdots\oplus T \zeta_1^{d-1}.
	\]

	Suppose $f_1,f_2,\ldots,f_d \in R$ are elements for which
	\[
	u=\sum_{j=1}^d a_j f_j
	\]
	is fixed by \(S_d\).  Then \(u\) is determined by \(f_1\).  Moreover,
	\(f_1\) must be fixed by the stabilizer of the index \(1\), which is the 
	symmetric group
	\(S_{\{2,\ldots,d\}}\).  Hence
	$ f_1\in R^{S_{\{2,\ldots,d\}}}$.
	Hence there are unique elements \(h_0,\ldots,h_{d-1}\in T\) such that
	\[
	f_1=\sum_{m=0}^{d-1}h_m\zeta_1^m.
	\]
	Since \(u\) is fixed by \(S_d\), for $j=2,3,\ldots,d$, the coefficient of \(a_j\) is obtained by
	applying a permutation taking \(1\) to $j$.  Therefore
	\[
	u
	=
	\sum_{m=0}^{d-1}h_m
	\sum_{j=1}^d a_j\zeta_j^m
	=
	\sum_{m=0}^{d-1}h_m\nu_m.
	\]
	This shows that \(\nu_0,\ldots,\nu_{d-1}\) generate the fixed point part over $T$.
	
	The same argument shows linear independence.  Indeed, if $h_0,h_1,\ldots,h_{d-1} \in T$ satisfy 
	\[
	\sum_{m=0}^{d-1}h_m\nu_m=0,
	\]
	then the coefficient of \(a_1\) is
	\[
	\sum_{m=0}^{d-1}h_m\zeta_1^m=0.
	\]
	But the expansion in the basis
	\[
	1,\zeta_1,\ldots,\zeta_1^{d-1}
	\]
	as a \(T\)-basis is unique.  Hence \(h_m=0\) for every \(m\), as required.
\end{proof}

\begin{lem}\label{lem:nu-recurrence}
	With the notation as in \cref{lem:one-a-module},
for every $m\geq d$,
\begin{enumerate}
	\item\label{eq:nu-recurrence} 
	$ \nu_m
	=\varepsilon_1\nu_{m-1}-\varepsilon_2\nu_{m-2}
	+\cdots+(-1)^{d-1}\varepsilon_d\nu_{m-d}$.
	\item $\nu_m$ is decomposable in
$\vH^*(P_d^{\mathrm{alg}}(S^{2n});\Q)^{S_d}$ for every $m\geq d$.
\end{enumerate}
\end{lem}

\begin{proof}
	For $j=1,2,\ldots,d$, the element $\zeta_j$ satisfies the equation
\[
		\zeta_j^d
		-\varepsilon_1\zeta_j^{d-1}
		+\varepsilon_2\zeta_j^{d-2}
		-\cdots+(-1)^d\varepsilon_d=0.
\]
Multiply this equation by $a_j\zeta_j^{m-d}$ and sum over
$j=1,\ldots,d$.  This gives \eqref{eq:nu-recurrence}.  Each term on the
right hand side is a product of the positive degree class $\varepsilon_j$ with
the positive degree class $\nu_{m-j}$, so $\nu_m$ is decomposable.
\end{proof}

For completeness, we also record why the fixed point ring is generated by the
classes which have just appeared.

\begin{prp}\label{lem:fixed-point-ring-generated}
	Let
	\[
	R=
	\Q[a_1,\ldots,a_d,\zeta_1,\ldots,\zeta_d]/
	(a_1^2,\ldots,a_d^2),
	\]
	with $S_d$ acting by simultaneously permuting the pairs
	$(a_j,\zeta_j)$. As in \eqref{eq:varepsilon-def}, let
	\[
	\varepsilon_j=e_j(\zeta_1,\ldots,\zeta_d)
	\; \text{ for } \;  j=1,2,\ldots,d,
	\]
	and, as in \eqref{eq:nu-m-def}, for \(m\geq 0\), let
	\[
	\nu_m=\sum_{j=1}^d a_j\zeta_j^m.
	\]
	Then the fixed point ring \(R^{S_d}\) is generated as a \(\Q\)-algebra by
	$
	\varepsilon_1,\ldots,\varepsilon_d$ and $\nu_0,\nu_1,\ldots,\nu_{d-1}$, and the indecomposable quotient of \(R^{S_d}\) is generated as a $\Q$-vector space by the
	images of these classes.
\end{prp}

\begin{proof}
	Let
	\[
	T=\Q[\varepsilon_1,\ldots,\varepsilon_d]
	=
	\Q[\zeta_1,\ldots,\zeta_d]^{S_d}.
	\]
	For \(c\geq 1\), set
	\[
	p_c=\sum_{j=1}^d \zeta_j^c.
	\]
	By Newton's identities (\cite[page 23, equations (2.11) and (2.11')]{Macdonald1995}), we have \(p_c\in T\) for every \(c\geq 1\).	
	The ring \(R\) has a \(\Q\)-basis consisting of monomials
	\[
	\bigl \{ a_{i_1}\cdots a_{i_r}
	\zeta_1^{b_1}\cdots \zeta_d^{b_d} \mid r = 0,1,\ldots , d \; , \;
	1 \leq i_1 < i_2 < \cdots < i_r \leq d  \; , \, b_1,b_2,\ldots, b_d \in \Nz \bigr \},
	\]
	where for $r=0$ the empty product is taken to be $1$.  As the elements $a_1,a_2,\ldots, a_d$ commute, the  group $S_d$ 
	permutes this basis.  Therefore the fixed point subspace is 
	spanned by sums of such monomials over each orbit.
	
	We now introduce a convenient family of orbit sums.  For 
	integers
	\(r,q\geq 0\), for nonnegative integers \(m_1,\ldots,m_r\), and for strictly positive
	integers \(c_1,\ldots,c_q\), define
	\[
	\Omega(m_1,\ldots,m_r;c_1,\ldots,c_q) = 
	\sum
	a_{i_1}\cdots a_{i_r}
	\zeta_{i_1}^{m_1}\cdots \zeta_{i_r}^{m_r}
	\zeta_{j_1}^{c_1}\cdots \zeta_{j_q}^{c_q},
	\]
	where the sum is over all ordered tuples
	\[
	(i_1,\ldots,i_r,j_1,\ldots,j_q)
	\]
	of pairwise distinct elements of \(\{1,\ldots,d\}\).  If \(r+q>d\), this sum
	is understood to be zero.  When \(q=0\), we write
	\(\Omega(m_1,\ldots,m_r;\varnothing)\) for this element; if $q=0$ and $r=0$, this is taken to mean $1$.  These elements
	clearly span \(R^{S_d}\).
	
	Let \(C\) be the subalgebra of \(R^{S_d}\) generated by
	$T$ and by $\nu_0,\nu_1,\nu_2,\ldots$.
	We claim that  $\Omega(m_1,\ldots,m_r;c_1,\ldots,c_q) \in C$ for all applicable tuples
	$m_1,\ldots,m_r$ and $c_1,\ldots,c_q$.
	For $r = 0$, those elements are all in $T$, so it is enough to consider the ones for which $r>0$. 
	The proof is by induction on $q$.
	
	If \(q=0\), then
	\[
	\Omega(m_1,\ldots,m_r;\varnothing)
	=
	\nu_{m_1}\nu_{m_2}\cdots\nu_{m_r}.
	\]
	We have
	\[
	\nu_{m_1}\nu_{m_2}\cdots \nu_{m_r}
	=
	\left(\sum_{j=1}^d a_{j}\zeta_{j}^{m_1}\right)
		\left(\sum_{j=1}^d a_{j}\zeta_{j}^{m_2}\right)
	\cdots
	\left(\sum_{j=1}^d a_{j}\zeta_{j}^{m_r}\right) .
	\]
	Expanding this product gives
	\[
	\nu_{m_1}\nu_{m_2}\cdots \nu_{m_r}
	=
	\sum_{(i_1,\ldots,i_r)\in \{1,\ldots,d\}^r}
	a_{i_1}\cdots a_{i_r}
	\zeta_{i_1}^{m_1}\cdots \zeta_{i_r}^{m_r}.
	\]
	We split this sum into the terms for which $i_1,\ldots,i_r$ are pairwise distinct
	and the terms for which at least two indices coincide:
	\[
	\begin{aligned}
		\nu_{m_1}\cdots \nu_{m_r}
		&=
		\sum_{\substack{(i_1,\ldots,i_r)\\ \text{pairwise distinct}}}
		a_{i_1}\cdots a_{i_r}
		\zeta_{i_1}^{m_1}\cdots \zeta_{i_r}^{m_r} \\
		&\qquad
		+
		\sum_{\substack{(i_1,\ldots,i_r)\\  i_\alpha=i_\beta \text{ for some } \alpha\neq\beta}}
		a_{i_1}\cdots a_{i_r}
		\zeta_{i_1}^{m_1}\cdots \zeta_{i_r}^{m_r}.
	\end{aligned}
	\]
	If $i_\alpha=i_\beta$ with $\alpha\neq\beta$, then the corresponding term contains
	the factor
	\[
	a_{i_\alpha}a_{i_\beta}=a_{i_\alpha}^2=0
	\]
	in $R$. Hence every term in the second sum is zero. Therefore
	\[
	\nu_{m_1}\cdots \nu_{m_r}
	=
	\sum_{\substack{(i_1,\ldots,i_r)\\ \text{pairwise distinct}}}
	a_{i_1}\cdots a_{i_r}
	\zeta_{i_1}^{m_1}\cdots \zeta_{i_r}^{m_r}.
	\]
	The right-hand side is exactly $\Omega(m_1,\ldots,m_r;\varnothing)$.
	 If $r>d$,
	there are no pairwise distinct $r$-tuples, and every term in the expansion of
	$\nu_{m_1}\cdots\nu_{m_r}$ has a repeated index, hence is zero, so the 
	argument still applies.
	Therefore
	\(\Omega(m_1,\ldots,m_r;\varnothing)\in C\).
	
	Now assume \(q\geq 1\), and write
	\[
	m=(m_1,\ldots,m_r),
	\qquad
	c'=(c_1,\ldots,c_{q-1}),
	\qquad
	c=c_q.
	\]
	Here \(c'\) denotes the empty tuple, written \(\varnothing\), when 
	\(q=1\).
	We have
	\[
	\begin{aligned}
		\Omega(m ;c')p_c
		&=
		\Omega( m ;c',c)                                      \\
		&\quad
		+\sum_{\nu=1}^r
		\Omega(m_1,\ldots,m_{\nu-1},m_\nu+c,m_{\nu+1},\ldots,m_r;c')                   \\
		&\quad
		+\sum_{\nu=1}^{q-1}
		\Omega(m ;c_1,\ldots,c_{\nu-1},c_\nu+c,c_{\nu+1},\ldots,c_{q-1}).
	\end{aligned}
	\]
	
	By the induction hypothesis, every term on the right except possibly
	\(\Omega( m; c',c)\) is in \(C\).  Also
	\(p_c\in T\subset C\) and $\Omega(m;c') \in C$, so the left hand side is in \(C\).  Hence
	$\Omega( m;c',c)\in C$.
	This completes the induction and proves $R^{S_d}=C$.
		
	By \cref{lem:nu-recurrence}, for $m\geq d$, the element $\nu_m$ is in the
	algebra generated by $\varepsilon_1,\ldots,\varepsilon_d$ and $ \nu_0,\ldots,\nu_{d-1}$. Therefore \(R^{S_d}\) is generated by these classes. 
	
	It remains to show that $\varepsilon_1,\ldots,\varepsilon_d$ and $ \nu_0,\ldots,\nu_{d-1}$ are in fact indecomposable. We can endow $R$ with a different grading, defined by assigning $\zeta_1,\ldots,\zeta_d$ degree $0$ and assigning $a_1,\ldots,a_d$ degree $1$. With this auxiliary grading, the elements $\varepsilon_1,\ldots,\varepsilon_d$ have degree $0$ and the elements $\nu_0,\ldots,\nu_{d-1}$ have degree $1$. This shows that if one of the elements $\varepsilon_1,\ldots,\varepsilon_d$ can be decomposed as a linear combination of products of the other elements among $\varepsilon_1,\ldots,\varepsilon_d$ and $ \nu_0,\ldots,\nu_{d-1}$, then, by isolating just the degree zero part, it can be decomposed as a linear combination of products of other elements from $\varepsilon_1,\ldots,\varepsilon_d$. This cannot happen, because the subalgebra generated by $\varepsilon_1,\ldots,\varepsilon_d$ is the polynomial ring in those variables. Now, if one the elements $ \nu_0,\ldots,\nu_{d-1}$, were decomposable, it could be written as a linear combination of the others with coefficients in $\Q[\varepsilon_1,\ldots,\varepsilon_d]$, which contradicts \cref{lem:one-a-module}.
\end{proof}

\subsection{Proof of the main theorem}
\begin{proof}[Proof of \cref{thm:rank-d-obstruction}]
	We first prove \eqref{thm:rank-d-obstruction_item_1}. 
Let
\[
       \beta\in \widetilde K^0(S^{2n})
\]
be the reduced Bott class.  Define
$a\in\widetilde{\vH}^{2n}(S^{2n};\Q)$ by
\begin{equation}\label{eq:Ch-u}
      a = \Ch(\beta).
\end{equation}
For $j=1,2,\ldots,d$, let \(a_j\) be the pullback of \(a\) to the \(j\)-th copy of \(S^{2n}\).

Let
\[
       \psi_{S^{2n},d} \colon C(S^{2n})\longrightarrow C(\Ea_{d,S^{2n}},K)
\]
be the universal homomorphism from Definition~\ref{dfn_general_theory_def_of_psi}.  As in \cref{lem_pullback_identification}, with $\beta$ in place of $\xi$, let
$\gamma\in RK^0(\Ea_{d,S^{2n}})$ (called $\zeta$ in \cref{lem_pullback_identification}) be the class corresponding to
$(\psi_{S^{2n},d})_*(\beta)\in RK_0(C(\Ea_{d,S^{2n}},K))$.

For $j=1,2,\ldots, d$, denote by $\rho_j\colon (S^{2n})^d\to S^{2n}$ the 
projection to the $j$-th  factor. We now have, in rational $K$-theory,
\begin{equation}\label{eq:pi-pullback-xi}
       (\pi^{\mathrm{alg}})^*(\gamma)
          =\sum_{j=1}^d \rho_j^*(\beta)\otimes [L_j] .
\end{equation}
 Applying the Chern character (\cite[Chapter V, Section 3, Theorem 3.23, pp.~282--283]{Karoubi}),
  with $\nu_m$ for $m=0,1,\ldots$ defined as in \eqref{eq:nu-m-def}, we get
\begin{equation}\label{eq:Ch-xi-ordered}
       (\pi^{\mathrm{alg}})^*\Ch(\gamma)
          =\sum_{j=1}^d a_j\exp(\zeta_j)
          =\sum_{m=0}^{\infty}\frac{1}{m!}\nu_m.
\end{equation}
Consequently the degree $2(n+m)$ component of $(\pi^{\mathrm{alg}})^*\Ch(\gamma)$ is
\begin{equation}\label{eq:degree-component-nu}
       \frac{1}{m!}\nu_m.
\end{equation}
There are no components in positive degrees below $2n$.

Now suppose, for contradiction, that
\[
       \varphi\colon C(S^{2n}) \to C(S^{2k},K)
\]
has $\rank(\varphi(1))=d$ and is injective on $K_0$.  Using \cref{lem:E_alg_homotopy_equivalence}, we may assume without loss of generality
that $\varphi \colon C(S^{2n}) \to C(S^{2k},M_{\infty})$.  As in \cref{lem_general_theory_h_vaprhi}, it 
determines a continuous map $h_{\varphi} \colon S^{2k}\to E_{d,S^{2n}}$.
By \cref{lem:E_alg_homotopy_equivalence}, we may assume that the codomain is 
$\Ea_{d,S^{2n}}$.
By \cref{lem_pullback_identification}, $h_\varphi^*(\gamma)=\varphi_*(\beta)$ in $K^0(S^{2k})\otimes\Q$.
Since $\varphi_*$ is injective on $K_0$ and $\beta$ is a generator of the 
reduced
group  $\widetilde{K}^0(S^{2n})$, the class $\varphi_*(\beta)$ is nonzero.  Because
$K^0(S^{2k})$ is torsion-free and the rational Chern character is injective on
$K^0(S^{2k})\otimes\Q$, it follows that
\begin{equation}\label{eq:nonzero-pulled-chern}
       h_\varphi^*( \Ch(\gamma) ) \neq 0
       \quad\hbox{in } \widetilde{\vH}^{2k}(S^{2k};\Q).
\end{equation}

We now show that this is impossible when
$k\notin\{n,n+1,\ldots,n+d-1\}$.  If $k<n$, then $\Ch(\gamma)$ has no component in
degree $2k$, since the reduced Bott class begins in degree $2n$.  Thus the
degree $2k$ component is zero, and hence decomposable.  By
\cref{cor:injectivity_Chern_class_criterion}, $\varphi_*$ cannot be injective
on $K_0$, a contradiction.

It remains to consider $k\geq n+d$.  Write $k=n+m$ with $m\geq d$.
We compute $\Ch(\gamma)$ on the algebraic model.  By
\eqref{eq:degree-component-nu}, its degree $2k=2(n+m)$ component pulls back to
the ordered algebraic model as a nonzero scalar multiple of $\nu_m$.  By
\cref{lem:nu-recurrence}, $\nu_m$ is decomposable for every
$m\geq d$.  By \cref{prp:cohomology-parameter-space-fixed-ring},
$\vH^*(\Ea_{d,S^{2n}};\Q)$ is identified with the $S_d$-fixed point subring
of the ordered algebraic model.  Hence the degree $2k$ component of
$\Ch(\gamma)$ is decomposable in $\vH^*(\Ea_{d,S^{2n}};\Q)$.  Therefore
\cref{cor:injectivity_Chern_class_criterion} again implies that $\varphi_*$ is
not injective on $K_0$, a contradiction.

Thus no such $\varphi$ exists when
$k\notin\{n,n+1,\ldots,n+d-1\}$.  

For \eqref{thm:rank-d-obstruction_item_2}, suppose $r \in \N$, and suppose  $\varphi\colon C(S^{2n},M_r)\to C(S^{2k},M_{rd})$ is unital. Let $e \in C(S^{2n},M_r)$ be a constant rank $1$ projection. Identify 
$e C(S^{2n},M_r) e \cong C(S^{2n})$. Write $\omega = \varphi|_{e C(S^{2n},M_r) e}$. Then $\omega(e)$ is a rank $d$ projection in $C(S^{2k},M_{rd})$, which can be thought of as a rank $d$ projection in $C(S^{2k},K)$. If $\varphi_*$ is injective on $K_0$, then so is $\omega_*$, which by  \eqref{thm:rank-d-obstruction_item_1} is not possible.
\end{proof}

\begin{cor}\label{cor:odd-sphere-obstruction}
	Let $n,d,k\in \N$.
	There exists a unital homomorphism
	\[
	\varphi\colon C(S^{2n-1}) \to C(S^{2k-1},M_d)
	\]
	such that $\varphi_*$ is injective on $K_1$ if and only if $k\in \{n,n+1,\ldots,n+d-1\}$.
\end{cor}

\begin{proof}
	Assume $k\notin \{n,n+1,\ldots,n+d-1\}$. If such a map existed, its suspension would give a unital homomorphism
	$C(S^{2n})\to C(S^{2k},M_d)$.  Under the suspension isomorphism, injectivity on
	$K_1(C(S^{2n-1}))$ becomes injectivity on the reduced $K_0$-summand of
	$C(S^{2n})$. Injectivity on all of $K_0$ is then clear. This contradicts
	\cref{thm:rank-d-obstruction}. The converse was discussed in the introduction.
\end{proof}

\section{Remarks and open problems}\label{Sec_remarks_and_problems}
\begin{rmk}
	Although our methods in the paper rely on understanding the structure of the cohomology ring of the homomorphism space $E_{d,X}$, the ideas leading to them were motivated by tools from rational homotopy theory. Indeed, using Sullivan minimal models and our computations of rational cohomology, one can show, for example, that  
	$\pi_{4n} (E_{2,S^{2n}} ) \otimes \Q = 0$ whenever $n>1$. If $\varphi \colon C(S^{2n}) \to C(S^{4n},K)$ is a homomorphism such that $\rank (\varphi(1)) = 2$, then the corresponding map 
	$h_{\varphi} \colon S^{4n} \to E_{2,S^{2n}}$ defines an element of $\pi_{4n} (  E_{2,S^{2n}} )$. One can show that if $\varphi_*$ is injective on $K_0$, then $[h_{\varphi}]$ is a non-torsion element of $\pi_{4n} (  E_{2,S^{2n}} )$, which cannot happen. (We do not claim that in general, if $n,d,k\in\N$ and 
	$k\notin\{n,n+1,\ldots,n+d-1\}$ then $\pi_{2k} (  E_{d,S^{2n}} ) \otimes \Q = 0$.) 
	We refer the reader to \cite{FHT} for more on rational homotopy theory. 
\end{rmk}

When we fix the matrix size, we can obtain an approximate version of \cref{thm:rank-d-obstruction}.
For a completely positive map $\varphi \colon A \to B$ between two C*-algebras, and for $l \in \N$, we write $\varphi^{(l)} \colon M_l(A) \to M_l(B)$ for the amplification. For a finite subset $F \subset A$ and $\varepsilon>0$, we say that $\varphi$ is $(F,\varepsilon)$-multiplicative if $\|\varphi(a)\varphi(b) - \varphi(ab) \| < \varepsilon$ for any $a,b \in F$. 

\begin{prp}\label{prp:fixed-rank-approximate-obstruction}
	Let $n,k,d,l\in\N$. Set $ B=C(S^{2k},M_d)$.
	Let
	$p\in M_l(C(S^{2n}))$ be a projection such that
	\[
	 [p]-\rank(p) \cdot [1_{ C(S^{2n}) }]
	\in\widetilde K_0(C(S^{2n}))
	\]
	is a generator. 
	Suppose that no unital homomorphism from $C(S^{2n})$ to $B$
	is injective on $K_0$.  Then there exist a finite set
	$F\subset C(S^{2n})$ and $\varepsilon>0$ such that the following holds.
	
	Let
	$\varphi\colon C(S^{2n}) \to C(S^{2k},M_d)$
	be a unital completely positive $(F,\varepsilon)$-multiplicative map.  
	Then $1/2$ is not 
	in the spectrum of 
	$\varphi^{(l)}(p)$, and the spectral
	projection
	\[
	q_\varphi=\chi_{(1/2,\infty)}\bigl(\varphi^{(l)}(p)\bigr)
	\]
	satisfies $[q_\varphi]-\rank(p) \cdot [1_B]=0$ in $\widetilde K_0(B)$.
\end{prp}

Let $X$ be a compact metrizable space. For $d\in\N$, let $\operatorname{UCP}_d(X)$
be the space of unital completely positive maps from $C(X)$ to $M_d$, with the topology of
pointwise norm convergence. With this topology, $\operatorname{UCP}_d(X)$ is a compact metrizable space. The following lemma is straightforward, so the proof is omitted. The space  $E_{d,S^{2n},d}$ is an ANR because it is a finite CW complex.

\begin{lem}\label{lem:fixed-rank-representation-space-anr}
	For every $d,n\in\N$, the space $E_{d,S^{2n},d}$ is a compact ANR, and it is closed in
	$\operatorname{UCP}_d(S^{2n})$.
\end{lem}

\begin{lem}\label{lem:fixed-rank-approximate-retraction}
	Let $n,d,l\in\N$. Let $p \in M_l(C(S^{2n}))$ be a projection. There are a finite set $F\subset C(S^{2n})$ and
	$\varepsilon>0$ with the following property.  If
	$\theta\colon C(S^{2n})\to M_d$ is a unital completely positive
	$(F,\varepsilon)$-multiplicative map, then there is a unital homomorphism
	$\rho\colon C(S^{2n})\to M_d$ such that $\|\theta^{(l)}(p)-\rho^{(l)}(p)\|<1/4$.
	
	Moreover, if $Y$ is a compact Hausdorff space and the maps
	$\theta_y\colon C(S^{2n})\to M_d$, for $y\in Y$, are a pointwise norm
	continuous family of UCP $(F,\varepsilon)$-multiplicative maps, then one can choose 
	a pointwise norm continuous family of
	homomorphisms $\rho_y$ with the above property. 
\end{lem}

\begin{proof}
	By \cref{lem:fixed-rank-representation-space-anr}, the closed subspace
	$E_{d,S^{2n},d}\subset \operatorname{UCP}_d(S^{2n})$ is a compact ANR.  Hence
	there exist an open neighborhood $V$ of $E_{d,S^{2n},d}$ in
	$\operatorname{UCP}_d(S^{2n})$ and a continuous retraction
	\[
	r \colon V\to E_{d,S^{2n},d}.
	\]
	The function
	\[
		\theta \mapsto
		\|\theta^{(l)}(p)-r (\theta)^{(l)}(p)\|
	\]
	is continuous on $V$ and vanishes on $E_{d,S^{2n},d}$. After replacing $V$ by a smaller open neighborhood of $E_{d,S^{2n},d}$
	we may assume that  $\|\theta^{(l)}(p)-r (\theta)^{(l)}(p)\| < 1/4$ for all $\theta\in V$.
	
	It remains to see that sufficiently multiplicative UCP maps land in this
	neighborhood.  Suppose not.  Choose a dense sequence $(f_m)_{m \in \N}$ in the unit ball
	of $C(S^{2n})$.  For every $m$ there would then be a UCP map
	\[
	\theta_m\colon C(S^{2n})\to M_d
	\]
	which is $(\{f_1,\ldots,f_m\},1/m)$-multiplicative but does not belong to
	$V$.  Because 
	$\operatorname{UCP}_d(S^{2n})$ is compact and metrizable, by passing to 
	a subsequence, we may assume without loss of generality that this sequence converges. 
	Let $\theta$ denote the limit. Then $\theta\in E_{d,S^{2n},d}$ but also $\theta \not \in V$, which is a contradiction.
	
	The last assertion follows because the retraction $r \colon V\to E_{d,S^{2n},d}$
	is continuous. 
\end{proof}

\begin{proof}[Proof of \cref{prp:fixed-rank-approximate-obstruction}]
	Let $F$ and $\varepsilon$ be as in
	\cref{lem:fixed-rank-approximate-retraction}.  For $y\in S^{2k}$ define
	\[
	\varphi_y=\ev_y\circ\varphi\colon C(S^{2n})\to M_d.
	\]
	Then $y\mapsto\varphi_y$ is a continuous map from $S^{2k}$ to
	$\operatorname{UCP}_d(S^{2n})$, and $\varphi_y$ is
	$(F,\varepsilon)$-multiplicative for every $y$.  Therefore the retraction from
	\cref{lem:fixed-rank-approximate-retraction} gives a continuous family of
	unital homomorphisms
	$\rho_y\colon C(S^{2n})\to M_d$ such that $\| ( \varphi_y )^{(l)}(p)-( \rho_y )^{(l)}(p)\|<1/4$.
	Define
	$
	\rho\colon C(S^{2n})\to C(S^{2k},M_d)$ by
	$\rho(a)(y)=\rho_y(a)$.
	Then $\rho$ is a unital homomorphism, and
	$\|\varphi^{(l)}(p)-\rho^{(l)}(p)\|<1/4$.
	
	In particular, $1/2$ is not in the spectrum of 
	$\varphi^{(l)}(p)$.  
	Now, $q_{\varphi}$ is homotopic to $\rho^{(l)}(p)$, 
	so $[q_{\varphi}]-\rank(p) \cdot [1_B]=0$, as required.
\end{proof}

Villadsen constructed simple C*-algebras that exhibited non-stable phenomena which appear in homogeneous C*-algebras, and indeed, Blackadar's motivation for asking his question was an earlier attempt to construct such examples. Though we showed that the maps Blackadar looked for do not exist, our obstruction itself presents a different kind of non-stable phenomenon. It seems natural, then, to see if this type of obstruction can survive in the simple setting. More concretely, the following, if it exists, would exhibit a new phenomenon.
\begin{qst}\label{qst:exotic_AH}
	Do there exist a simple unital AH algebra $A$ and $n \in \N$ for which there exists
	an injective homomorphism of ordered $K_0$-groups  
	\[
	\alpha \colon K_0(C(S^{2n})) \to K_0(A)
	\]
	 with $\alpha([1_{C(S^{2n})}]) = [1_A]$ 
	but such that there is no unital homomorphism $\varphi \colon C(S^{2n}) \to A$ with $\varphi_* = \alpha$?
\end{qst}
As a first step, one would likely need a better version of \cref{prp:fixed-rank-approximate-obstruction}, which is uniform in $d$. Concretely, we could ask the following question.
\begin{qst}
	Let $n,l \in \N$, and let $p \in M_l( C(S^{2n})) $ be a projection such that $[p] - \rank(p)\cdot [1_{C(S^{2n})}]$ is a generator of $\widetilde{K}_0( C( S^{2n} ))$. 
	
	 Let $(k_m)_{m \in \N}$ and $(d_m)_{m \in \N}$ be increasing sequences of natural numbers such that no unital homomorphism from $C(S^{2n})$ to $B_m = C(S^{2k_m},M_{d_m})$ is injective on $K_0$. Do there exist a finite subset $F \subset  C( S^{2n} )$ and $\varepsilon > 0$ such that for any $m$ and for any unital completely positive $(F,\varepsilon)$-multiplicative map $\varphi \colon C( S^{2n} ) \to B_m$, the spectrum of $\varphi^{(l)}(p)$ does not contain $1/2$, and the spectral projection $q = \chi_{(1/2,\infty)}(\varphi^{(l)}(p))$ is stably equivalent to a trivial projection?
\end{qst}

\bibliographystyle{plain}
\bibliography{bibliography}

\end{document}